\documentclass{amsart} 
\usepackage{graphicx,amssymb,enumerate}
\usepackage[section]{algorithm}
\usepackage[hidelinks]{hyperref}

\newtheorem{theorem}{Theorem}[section] \newtheorem{lemma}{Lemma}[section]

\theoremstyle{definition} \newtheorem{remark}{Remark}[section]
\newtheorem{definition}{Definition}[section]

\title[Randomized mixed H\"older function approximation] {Randomized mixed
H\"older function approximation \\ in higher-dimensions}

\author[N.~F.~Marshall]{Nicholas F. Marshall} \address{Department of
Mathematics, Yale University, New Haven, CT 06511, USA} 
\email{nicholas.marshall@yale.edu}

\thanks{N.F.M. is supported by NSF DMS-1903015.}

\begin{document}

\begin{abstract} 
The purpose of this paper is to extend the result of
\href{https://arxiv.org/abs/1810.00823}{\texttt{arXiv:1810.00823}} to mixed
H\"older functions on $[0,1]^d$ for all $d \ge 1$. In particular, we prove that
by sampling an $\alpha$-mixed H\"older function $f : [0,1]^d \rightarrow
\mathbb{R}$ at  $\sim \frac{1}{\varepsilon} \left(\log \frac{1}{\varepsilon}
\right)^d$ independent uniformly random points from $[0,1]^d$, we can construct
an approximation $\tilde{f}$ such that 
$$ 
\|f - \tilde{f}\|_{L^2} \lesssim \varepsilon^\alpha \left(\log
\textstyle{\frac{1}{\varepsilon}} \right)^{d-1/2},
$$
with high probability.  \end{abstract} \maketitle

\section{Introduction} \subsection{Introduction} 
A function $f : [0,1]^2 \rightarrow \mathbb{R}$ is $(c,\alpha)$-mixed
H\"older if
$$
|f(x',y) - f(x,y)| \le c | x' - x|^\alpha, \quad
|f(x,y') - f(x,y)| \le c | y' - y|^\alpha,
$$
\quad \text{and} \quad
$$
|f(x',y') - f(x',y) - f(x,y') + f(x,y)| \le c (|x'-x| |y'-y|)^\alpha,
$$
for all $x,x',y,y' \in [0,1]$. This definition extends to higher
dimensions as follows. Let $R = I_1 \times \cdots \times I_d \subseteq [0,1]^d$
be the $r$-dimensional box formed by taking the Cartesian product of $r$
intervals $I_{i_1} = [a_{i_1},a_{i_1}+h_{i_1}), \ldots, I_{i_r} =
[a_{i_r},a_{i_r}+h_{i_r})$ and $d-r$ singleton sets
$I_{i_{r+1}} = \{a_{i_{r+1}}\},\ldots,I_{i_d} = \{a_{i_d}\}$, where
$i_1,\ldots,i_d$ is a permutation of $1,\ldots,d$. We define the discrete mixed
difference $\delta_R f$ by
\begin{equation} \label{deltaR}
\delta_R f = ( \delta_{h_{i_1}} \cdots \delta_{h_{i_r}} f)(a_1,\ldots,a_d),
\end{equation}
where $\delta_{h_j}$ is the $j$-th variable discrete difference operator defined
by 
$$
(\delta_{h_j} f)(x_1,\ldots,x_d) = f(x_1,\ldots,x_{j-1},x_j +
h_j,x_{j+1},\ldots,x_d) - f(x_1,\ldots,x_d).
$$

\begin{definition}  \label{defmixed}
We say a function $f : [0,1]^d \rightarrow \mathbb{R}$ is $(c,\alpha)$-mixed
H\"older if 
\begin{equation} \label{mixed}
|\delta_R f| \le c |R|^\alpha, 
\end{equation}
for all $r \le d$ dimensional
boxes $R$, where $|R|$ is the $r$-dimensional measure of $R$.
\end{definition}

For example, if $f : [0,1]^d \rightarrow \mathbb{R}$ satisfies the mixed
derivative condition $\left|\partial^\beta f\right| \le c$ on $[0,1]^d$ for all
multi-indices $\beta \in \{0,1\}^d$, where $\partial^\beta =
\partial^{\beta_1}_{x_1} \cdots \partial^{\beta_d}_{x_d}$, then it follows from
the mean value theorem that $f$ is $(c,1)$-mixed H\"older. Thus, the H\"older
continuity condition is to the condition that $|\partial_{x_j} f|$ is bounded
for all indices $j \in \{1,\ldots,d\}$  as the mixed H\"older condition is to
the condition that $|\partial^\beta f|$ is bounded for all multi-indices $\beta
\in \{0,1\}^d$.

\subsection{Motivation} \label{motivation}
We say that a function $f : [0,1]^d \rightarrow \mathbb{R}$ is
$(c,\alpha)$-H\"older continuous if
$$
|f(x) - f(y)| \le c \|x - y\|^\alpha,
$$
for all $x,y \in [0,1]^d$, where $\|\cdot\|$ denotes the Euclidean norm. If a
function $f : [0,1]^d \rightarrow \mathbb{R}$ is only known to be
$(c,\alpha)$-H\"older continuous, then $\sim (\frac{1}{\varepsilon})^{d}$ samples of the
function are needed to construct an approximation $\tilde{f}$ such that 
$$
\|f - \tilde{f}\|_{L^\infty} \lesssim \varepsilon^\alpha,
$$
where the implicit constant only depends on the constant $c > 0$. Moreover,
sampling $\sim (\frac{1}{\varepsilon})^d$ function values is necessary to
achieve this approximation error in $L^p$ for any fixed $p \ge 1$. The fact that
the number of required samples $\sim (\frac{1}{\varepsilon})^{d}$ grows
exponentially with the dimension $d$ for any fixed $\varepsilon > 0$ is an
example of the curse of dimensionality. Even in moderate dimensions, such
sampling requirements may be intractable in practical situations. 

The mixed H\"older condition strengthens the H\"older condition by
requiring that the mixed difference of $f$ with respect to each box is
controlled by the measure of the box; more precisely, it requires that
\begin{equation} \label{geo}
|\delta_R f| \lesssim |R|^\alpha,
\end{equation}
for all $r \le d$ dimensional boxes $R$. We will see below that this stronger
geometric condition makes it possible to beat the curse of dimensionality.

We remark that while this paper focuses on real-valued mixed H\"older functions
defined on $[0,1]^d$, the geometric condition \eqref{geo} is well-defined for
Banach space valued functions defined on a product of metric spaces;  extensions
of the results of this paper to more abstract settings are discussed in \S
\ref{discussion}.

We can construct an example of a mixed H\"older function on
$[0,1]^d$ by taking the product of $d$ H\"older continuous functions on
$[0,1]$.  More precisely, if $g_1,\ldots,g_d : [0,1] \rightarrow \mathbb{R}$ are
each $\alpha$-H\"older continuous, and we define $f : [0,1]^d \rightarrow
\mathbb{R}$ by
\begin{equation} \label{product}
f(x_1,\ldots,x_d) = g_1(x_1) \cdots g_d(x_d),
\end{equation}
for $(x_1,\ldots,x_d) \in [0,1]^d$, then it follows that $f$ is $\alpha$-mixed
H\"older. Moreover, if $f : [0,1]^d \rightarrow \mathbb{R}$ is a linear
combination of products of the form \eqref{product}, then it follows that $f$ is
$\alpha$-mixed H\"older. In general, \eqref{geo} can be viewed as enforcing a
local version of these product structures on a function. The approximation
theory of functions with this type of local product regularity was first
developed by Smolyak in 1963, see \cite{Smolyak1963}. In particular, Smolyak
developed an approximation method that involves using function values at the
center of dyadic boxes in $[0,1]^d$, where a dyadic box in $[0,1]^d$ is a
Cartesian product of $d$ dyadic intervals. 

\begin{lemma}[Smolyak] \label{lem1} 
Suppose that $f : [0,1]^d \rightarrow \mathbb{R}$ is a $(c,\alpha)$-mixed
H\"older function that is sampled at the center of all dyadic boxes of measure
at least $\varepsilon$ contained in $[0,1]^d$, which is a set of $\sim
\frac{1}{\varepsilon} (\log \frac{1}{\varepsilon})^{d-1}$ points. Then, using
the function values of $f$ at these
points we can compute an approximation $\tilde{f}$ such that 
$$
\| f - \tilde{f} \|_{L^\infty} \lesssim \varepsilon^{\alpha} (\log
{\textstyle \frac{1}{\varepsilon}} )^{d-1},
$$ 
where the implicit constant
only depends on the constant $c >0$, and the dimension $d \ge 1$.
\end{lemma}

We state and prove a constructive version of this result later in the paper,
see Lemma \ref{lem2}. The method of Smolyak has inspired a family of so called
sparse grid computational methods, which are surveyed by  Bungartz and Griebel
\cite{BungartzGriebel2004}. A sparse grid is a set formed from the center points
of all dyadic boxes whose measure exceeds a fixed threshold: while the standard
grid $(\varepsilon \mathbb{Z}^d) \cap [0,1]^d$ has $\sim
\left(\frac{1}{\varepsilon} \right)^d$ points, a sparse grid for $[0,1]^d$ with
threshold $\varepsilon$ contains just $\sim \frac{1}{\varepsilon} \left(\log
\frac{1}{\varepsilon} \right)^{d-1}$ points. The theory of sparse grids has been
developed by several authors including \cite{ GerstnerGriebel1998,
GriebelHamaekers2014, KnapekKoster2002}. The method of Smolyak also
inspired work by Str\"omberg who observed that Smolyak's result can be
reformulated as an approximation result based on tensor wavelets
\cite{Stromberg1998}. Our approach to approximating mixed H\"older functions is
inspired by Str\"omberg's work, and the connection of Smolyak's method to tensor
wavelet systems. 

This paper builds upon previous work \cite{Marshall2018b} by the
author, which established a randomized approximation result for mixed H\"older
functions in two-dimensions. In particular, the main result of
\cite{Marshall2018b} is that an $\alpha$-mixed H\"older function on $[0,1]^2$
can be approximated to error $\lesssim \varepsilon^\alpha (\log
\frac{1}{\varepsilon})^{3/2}$ in $L^2$ with high probability using $\sim
\frac{1}{\varepsilon} (\log \frac{1}{\varepsilon})^2$ uniformly random
samples of the function.

\subsection{Main result} \label{mainresult} 
The purpose of this paper is to extend the result of \cite{Marshall2018b} to
higher dimensions. To clarify notation, we write $f \lesssim g$ to denote that
$f \le C g$ for some implicit constant $C > 0$, and write $f \sim g$ to denote
that $f \lesssim g$ and $g \lesssim f$.  The following theorem  holds for all
dimensions $d \ge 1$.

\begin{theorem} \label{thm1} Let $f :[0,1]^d \rightarrow \mathbb{R}$ be
a $(c,\alpha)$-mixed H\"older function, and $X_1,\ldots,X_n$ be 
$n$ independent uniformly random points from $[0,1]^d$. If $n \ge
c_1 \frac{1}{\varepsilon} (\log \frac{1}{\varepsilon})^d$, then using the
locations of these points and the function values $f(X_1),\ldots,f(X_n)$ we can
construct an approximation $\tilde{f}$ such that
$$
\|f - \tilde{f}\|_{L^2}
\lesssim \varepsilon^{\alpha} (\log {\textstyle \frac{1}{\varepsilon}} )^{d-1/2},
$$
with probability at least $1 - \varepsilon^{c_1 -2 \alpha}$, where the implicit
constant only depends  on the constants $c_1 > 0$ and $c > 0$, and the dimension
$d \ge 1$.
\end{theorem}

The construction of $\tilde{f}$ and the proof of Theorem \ref{thm1} are
detailed in \S \ref{proofofmain}. In the following, we make three
remarks related to Theorem \ref{thm1}.

\begin{remark}[Computational cost]\label{compcost}
The cost of constructing $\tilde{f}$  is $\sim \frac{1}{\varepsilon} (\log
\frac{1}{\varepsilon})^{2d-1}$ operations of pre-computation, and then $\sim
(\log \frac{1}{\varepsilon})^{d-1}$ operations for each point evaluation of
$\tilde{f}$, see \S \ref{compcostproof}. Additionally, after pre-computation, in
$\sim \frac{1}{\varepsilon}$ additional operations we can compute an
approximation $\tilde{I}$ of the integral of $f$ that satisfies
$$
\left| \int_{[0,1]^d} f(x) dx - \tilde{I} \right| \lesssim \varepsilon^\alpha
(\log {\textstyle \frac{1}{\varepsilon}} )^{d-1/2},
$$
with high probability; this error bound follows from
Theorem \ref{thm1}, and the claim that $\tilde{I}$ can be computed in $\sim
\frac{1}{\varepsilon}$ operations after the pre-computation is justified in \S
\ref{compcostproof}. 
\end{remark}

\begin{remark}[Spin Cycling] \label{spincycling} 
Using random samples makes it possible to  perform spin cycling without
re-sampling function values.  Spin cycling is a technique for refining wavelet
based approximations that involves averaging over shifts of the underlying
dyadic structure.
Suppose that $f :\mathbb{T}^d \rightarrow \mathbb{R}$ is a $\alpha$-mixed
H\"older function on the torus $\mathbb{T}^d$, $\{\gamma_1,\ldots,\gamma_s\}
\subset \mathbb{T}^d$ is a fixed set of shifts, and $\{X_1,\ldots,X_n\} \subset
\mathbb{T}^d$ is a fixed set of sample points. Let $f_\gamma(x) := f(x-\gamma)$,
and define $f^s$ by
\begin{equation} \label{shiftavg}
f^s(x) = \frac{1}{s} \sum_{i = 1}^s
\tilde{f}_{\gamma_i}(x+\gamma_i),
\end{equation}
where $\tilde{f}_\gamma$ is the approximation of $f_\gamma$ using the
method of Theorem \ref{thm1} with the points and function values
$$
\{X_j + \gamma\}_{j=1}^n \quad \text{and} \quad 
\quad \{f_\gamma(X_j+\gamma)\}_{j=1}^n.
$$ 
Since $f_{\gamma}(X_j + \gamma) = f(X_j)$, constructing this approximation does
not require sampling additional function values. Moreover, since the
construction of $\tilde{f}$ depends on the relation between the sample points
and dyadic decomposition of $[0,1]^d$, see \S \ref{proofofmain}, in general
$\tilde{f}_{\gamma}(x+\gamma)$ will differ from $\tilde{f}(x)$. The act of
aligning and averaging these approximations via \eqref{shiftavg} is called spin
cycling. When $n$ is chosen large enough in terms of $\varepsilon$ and $s$, each
function $\tilde{f}_\gamma(x+\gamma)$ will approximate $f(x)$ at the rate of
Theorem \ref{thm1} with high probability; empirically, averaging these
approximations has been shown to remove method artifacts, see \S 4.1 of
\cite{Marshall2018b}.
\end{remark}

\begin{remark}[Possible extensions]
It may be possible to generalize the approach of this paper to function classes
that have more regularity. The proof of Theorem \ref{thm1} is constructive and
describes a randomized method of approximating a mixed H\"older function using a
linear combination of indicator functions of dyadic boxes, which is equivalent
to using a linear combination of tensor Haar wavelets. Part of our motivation
for studying this relatively low regularity approximation problem is that we are
able to isolate some of the underlying geometric issues.  By using a function
class corresponding to smoother wavelets it may be possible to extend the
approach of this paper to develop randomized methods of approximating functions
with higher levels of regularity that have a similar local product structure to
mixed H\"older functions. In addition to extending the results of this paper to
approximate smoother functions, it may also be possible to strengthen the result
of Theorem \ref{thm1}. In particular, numerical evidence suggests that it may be
possible to remove some of the factors of $\log \frac{1}{\varepsilon}$ in
Theorem \ref{thm1}, and suggests that a similar approximation rate may hold in
$L^\infty$, see \S \ref{numericalexample}.  
\end{remark}

\subsection{Organization}
The remainder of the paper is organized as follows. In \S \ref{preliminaries} we
set notation and give mathematical preliminaries. In \S \ref{embedding} we
state an important definition and establish two key lemmas. In \S
\ref{proofofmain} we prove Theorem \ref{thm1}. In \S \ref{algoandex} we give
implementation details and a numerical example in three dimensions.  Finally, in
\S \ref{discussion} we discuss the results of this paper and possible
generalizations.

\section{Preliminaries} 
\label{preliminaries} 
\subsection{Notation} \label{notation} 
Let $\mathcal{M}^\alpha([0,1]^d)$ denote the space of all real-valued
$\alpha$-mixed H\"older functions on $[0,1]^d$. We say that $c > 0$ is the mixed
H\"older constant of $f \in \mathcal{M}^\alpha([0,1]^d)$ if $f$ is
$(c,\alpha)$-mixed H\"older on $[0,1]^d$.  Let $\mathcal{D}$ be the set of
dyadic intervals in $[0,1]$; more precisely, 
$$ 
\mathcal{D} := \left\{ \big[ (j-1) 2^{-k}, j 2^{-k} \big) \subset \mathbb{R} : k
\in \mathbb{Z}_{\ge 0} \wedge j \in \{1,\ldots,2^k\} \right\}, 
$$ 
and let $\mathcal{D}_d$ be the set of all dyadic boxes contained in $[0,1]^d$,
that is,
$$ 
\mathcal{D}_d := \left\{ I_1 \times \cdots
\times I_d \subset \mathbb{R}^d : I_1,\ldots,I_d \in \mathcal{D} \right\}.  
$$
We claim that the number of dyadic boxes of measure $2^{-m}$ in $[0,1]^d$ is
\begin{equation} \label{count}
\# \{ R \in \mathcal{D}_d : |R|=2^{-m} \} = 2^m { m + d - 1 \choose d - 1}.
\end{equation}
Indeed, suppose $R = I_1 \times \cdots \times I_d \in \mathcal{D}_d$ has
dimensions $2^{-i_1} \times \cdots \times 2^{-i_d}$. Since $R$ is a dyadic box
contained in $[0,1]^d$, it follows that $i_1,\ldots,i_d \in \mathbb{Z}_{\ge 0}$.
Moreover, if $R$ has measure $|R|=2^{-m}$ we must have
$$ 
i_1 + \cdots + i_d = m.
$$
Thus, the number of different shapes of dyadic boxes that have measure $2^{-m}$
and are contained in $[0,1]^d$ is 
\begin{equation} \label{count2}
\# \{(i_1,\ldots,i_d) \in \mathbb{Z}_{\ge 0}^d : i_1 + \cdots +
i_d = m \} = { m + d -1 \choose d-1}.
\end{equation}
Since for each fixed shape there are $2^m$ dyadic boxes of measure $2^{-m}$ 
contained in $[0,1]^d$, the counting formula \eqref{count} follows.  As an
example we illustrate all dyadic boxes of measure $2^{-m}$ in $[0,1]^3$ in
Figure
\ref{fig02}.  
\begin{figure}[h!] \centering
\includegraphics[width=.4\textwidth]{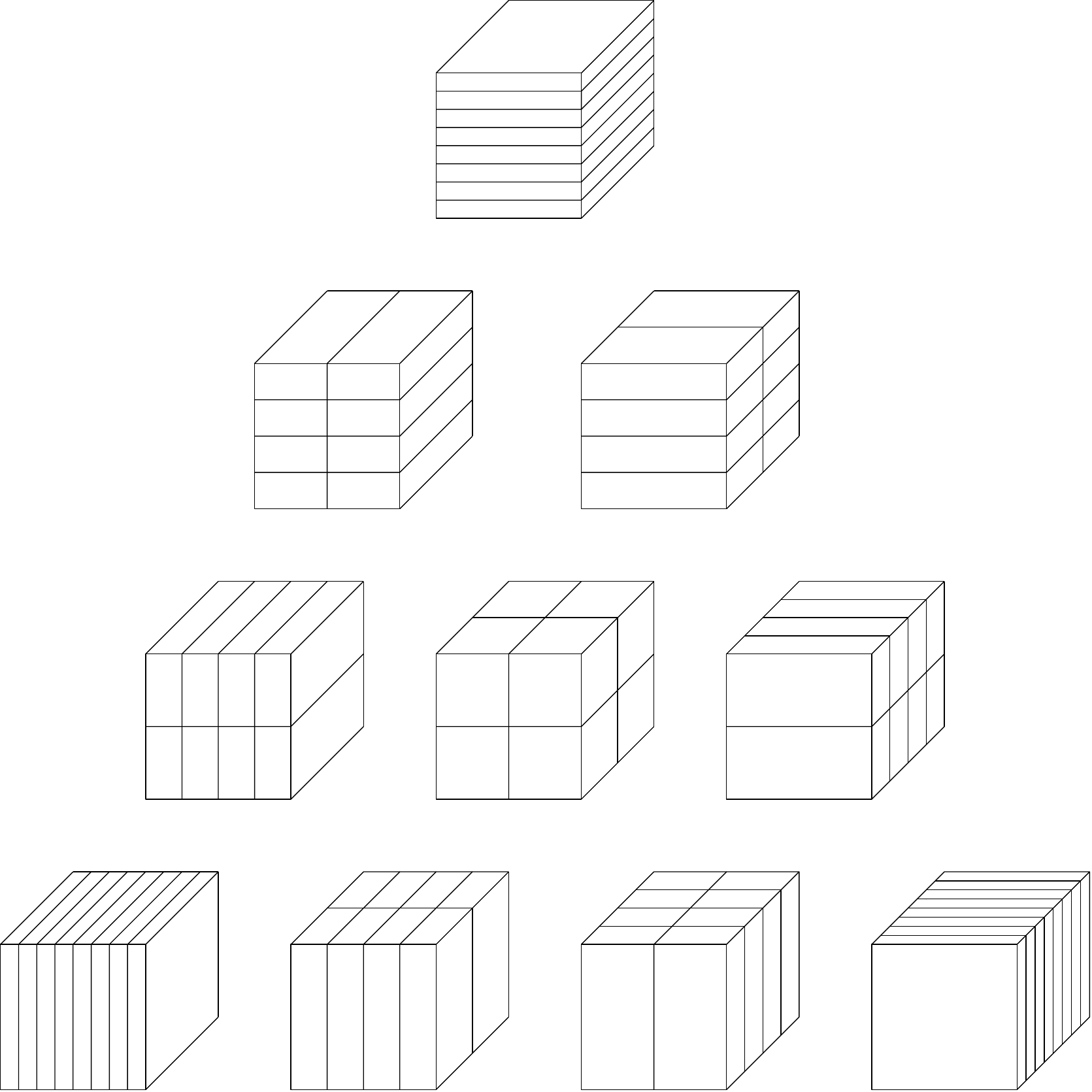} 
\caption{The set of $10 \cdot 2^3$ dyadic boxes of measure $2^{-3}$ in
$[0,1]^3$.} \label{fig02} 
\end{figure}

Recall that Lemma \ref{lem1} involves sampling function values at the
center of dyadic boxes of measure at least $\varepsilon$ in $[0,1]^d$. From the
above considerations, it follows that 
$$ 
\#\{ R \in \mathcal{D}_d : |R| \ge \varepsilon \} \sim {\textstyle
\frac{1}{\varepsilon}} (\log {\textstyle \frac{1}{\varepsilon}})^{d-1}, 
$$ 
that is,
the number of dyadic boxes of measure at least $\varepsilon$ in $[0,1]^d$  is
$\sim \frac{1}{\varepsilon} (\log\frac{1}{\varepsilon})^{d-1}$.

\subsection{Smolyak's Lemma} \label{deterministic} Recall that Lemma \ref{lem1}
says that if we sample values from a function $f \in \mathcal{M}^\alpha([0,1]^d)$ at the
center of all dyadic boxes of measure at least $\varepsilon$, then we can
compute an approximation $\tilde{f}$ such that $\|f - \tilde{f}\|_{
L^\infty} \lesssim \varepsilon^{\alpha} (\log \frac{1}{\varepsilon})^{d-1}$. The
following is a more precise version of Lemma \ref{lem1}.

\begin{lemma}[Smolyak] \label{lem2} Suppose that $f: [0,1]^d \rightarrow
\mathbb{R}$ is a $(c,\alpha)$-mixed H\"older function.  Fix $x \in [0,1]^d$, and
let $f_{\vec{i}}$ be the value of $f$ at the center of the dyadic box that
contains $x$, and has dimensions $2^{-i_1} \times \cdots \times 2^{-i_d}$, where
$\vec{i} = (i_1,\ldots,i_d)$. Then, for $m \ge 1$ we have 
\begin{equation} \label{smolyakeq}
\left| f(x) - \sum_{k=0}^{\min\{d-1,m\}} (-1)^k {d -1 \choose k}
\sum_{i_1+\cdots+i_d = m-k} f_{\vec{i}} \right| \lesssim 2^{-\alpha m} m^{d-1}, 
\end{equation}
where the implicit
constant only depends on the mixed H\"older constant $c > 0$ and the dimension
$d \ge 1$. 
 \end{lemma}

We remark that an elementary proof of this lemma for the case $d=2$ can be found
in \cite{Marshall2018b}, while the general $d$-dimensional proof is given
below.

\begin{proof}[Proof of Lemma \ref{lem2}] Fix $x \in [0,1]^d$, and let $\vec{m}
=(m,\ldots,m)$ denote the $d$-dimensional vector whose entries are all $m$.
Recall that $f_{\vec{i}}$ denotes the value of $f$ at the center of the dyadic
box that contains $x$ and has dimensions $2^{-i_1} \times \cdots \times 2^{-i_d}$.  Thus $f_{\vec{m}}$ is the value of $f$ at the center of the dyadic
box that contains $x$ and has dimension $2^{-m} \times \cdots \times 2^{-m}$.
Since $\alpha$-mixed H\"older functions are $\alpha$-H\"older continuous we
conclude that 
$$ 
\left| f(x) - f_{\vec{m}} \right| \lesssim 2^{-\alpha
m}.
$$ 
Thus, by the triangle inequality, in order to establish Lemma \ref{lem2} it
suffices to show that 
$$
\left| f_{\vec{m}} - \sum_{k = 0}^{\min \{d-1,m\}} (-1)^k {d -1 \choose k}
\sum_{i_1 + \cdots + i_d = m - k} f_{\vec{i}} \right| \lesssim 2^{-\alpha m}
m^{d-1} .  
$$ 
We will establish this inequality by inducting on the dimension $d$.  The base
case $d=1$ of the induction is trivial, so we assume that this inequality holds
in dimension $d$, and will show it holds in dimension $d+1$. Fix $x \in
[0,1]^{d+1}$, and let $f_{\vec{i},i_{d+1}}$ denote the value of $f$ at
the center of the dyadic box that has dimensions $2^{-i_1} \times \cdots \times
2^{-i_d} \times 2^{-i_{d+1}}$ and contains $x$.  We want to show that 
$$ 
\left|
f_{\vec{m},m} - \sum_{k=0}^{\min\{d,m\}} (-1)^k {d \choose k}
\sum_{i_1+\cdots+i_{d+1} = m-k}
f_{\vec{i},i_{d+1}} \right| \lesssim 2^{-\alpha m} m^{d}.
$$ 
First we write
$f_{\vec{m},m}$ as a telescopic series 
\begin{equation} \label{tel}
f_{\vec{m},m} = \sum_{j=1}^m \left( f
_{\vec{m},j} - f_{\vec{m},j-1} \right) + f_{\vec{m},0}.  
\end{equation}
The key observation
in the proof of this lemma is that when $j$ is fixed, the inductive hypothesis
can be applied to $g_{\vec{m}} = 2^{\alpha j} ( f_{\vec{m},j} -
f_{\vec{m},j-1})$ to conclude that 
\begin{multline}  \label{keysmolyak}
\left| \big( f _{\vec{m},j} - f_{\vec{m},j-1} \big)
\vphantom{\sum_{k=0}^{\min\{d-1,m-j\}} }\right.\\ \left.  -
\sum_{k=0}^{\min\{d-1,m-j\}} (-1)^k {d - 1 \choose k} \sum_{i_1+\cdots+i_d =
m-j-k} \left(f_{\vec{i},j} - f_{\vec{i},j-1} \right) \right| \lesssim 
2^{-\alpha m} m^{d-1}.  
\end{multline} 
Indeed, when $j$ is fixed, $g_{\vec{m}} =
2^{\alpha j} ( f_{\vec{m},j} - f_{\vec{m},j-1})$ can be viewed as having been
generated from a real-valued $(c,\alpha)$-mixed H\"older function on $[0,1]^d$.
Combing \eqref{tel} and \eqref{keysmolyak} yields 
\begin{multline*}
\left|f _{\vec{m},m} -  \left( \sum_{j=1}^m  \sum_{k=0}^{\min\{d-1,m-j\}} (-1)^k
{d - 1 \choose k} \sum_{i_1+\cdots+i_d = m-j-k} \left(f_{\vec{i},j} -
f_{\vec{i},j-1} \right) \right. \right) \\ \left.  -
\sum_{k=0}^{\min\{d-1,m\}} (-1)^k {d - 1 \choose k} \sum_{i_1+\cdots+i_d =
m-k} f_{\vec{i},0} \right| \lesssim m^{d} 2^{-\alpha m}.
\end{multline*}
We claim that, up to rearranging terms, the proof is complete.
Indeed, collecting the terms $f_{\vec{i},j}$ such that $k = 0$ and $i_1 + \cdots
+ i_d + j = m$  gives 
$$
\sum_{j=1}^m  \sum_{i_1+\cdots+i_d = m-j} f_{\vec{i},j} + \sum_{i_1+\cdots+i_d =
m} f_{\vec{i},0}  = \sum_{i_1+\cdots+i_d + i_{d+1} = m} f_{\vec{i},i_{d+1}}.  
$$
Moreover, collecting the terms $f_{\vec{i},j}$ such
that $k \ge 1$, $j \ge 1$, and $i_1 + \cdots + i_{d} + j = m -k$ gives $$ (-1)^k
{d-1 \choose k} f_{\vec{i},j} - (-1)^{k-1} {d - 1 \choose k-1} f_{\vec{i},j} =
(-1)^k {d \choose k} f_{\vec{i},j}; $$ summing these representations over
$j$ and $k$ gives the desired rearrangement.  \end{proof}

\begin{remark} 
Applying Lemma \ref{lem2} to the constant function gives the identity
$$ 1 =
\sum_{k = 0}^{\min\{d-1,m\}} (-1)^k {d -1 \choose k} {m - k + d - 1 \choose d -
1}, $$ 
which holds for all $d,m \ge 1$. 
\end{remark}

\subsection{Randomized Kaczmarz}
In addition to Smolyak's approximation method, we will use an result of Strohmer
and Vershynin \cite{StrohmerVershynin2009} about the randomized Kaczmarz
algorithm. The following lemma is a special case of the main result
of \cite{StrohmerVershynin2009}.

\begin{lemma}[Strohmer, Vershynin] \label{randomkaczmarz}
Let $A$ be an $N \times p$ matrix where $N \ge p$ whose rows are of equal
magnitude, and let $A w = b$ be a consistent linear system.  Suppose that
indices $I_1,I_2,\ldots$ are chosen independently and uniformly at random from
$\{1,\ldots,N\}$, and let an initial vector $w_0 \in \mathbb{R}^p$ be given. For
$n = 1,2,\ldots$ define 
$$
w_n := w_{n-1} +  \frac{b_{I_n} -\langle a_{I_n}, w_{n-1} \rangle
}{\|a_{I_n}\|_2^2}
a_{I_n}, 
$$
where $a_j$ denotes the $j$-th row of $A$, and $b_j$ denotes the $j$-th entry of
$b$. Then,
$$
\mathbb{E} \|w_n - w\|^2_2 \le (1 - \kappa^{-2})^n \|w_0 -
w\|_2^2,
$$
for $n = 1,2,\ldots$, where $\kappa^2 := \sum_{j=1}^p \sigma_j^2/\sigma_p^2$,
and where $\sigma_1,\ldots,\sigma_p$ are the singular values of $A$.  
\end{lemma}

Lemma \ref{randomkaczmarz} together with the embedding defined in \S
\ref{embedding} are the main ingredients of the proof of Theorem \ref{thm1}. In
particular, we use the randomized Kaczmarz algorithm to solve a noisy linear
system $A x \approx b + \epsilon$. The error analysis of the randomized Kaczmarz
algorithm for noisy linear systems was first performed by Needell
\cite{Needell2010}.  We require a slightly modified version of the result of
Needell so we use Lemma \ref{randomkaczmarz} directly in a similar argument to
that in \cite{Needell2010}.

\section{Embedding points} \label{embedding}
\subsection{Summary} In this section, we define an embedding $\Psi$
of points in $[0,1]^d$ into a high dimensional Euclidean space $\mathbb{R}^p$.
We want the embedding to have two properties. First, we do not want the
embedding to contain redundant information. In particular, we want the
coordinates of the embedding to be uncorrelated.  Second, the embedding should be defined such that mixed
H\"older functions can be approximated by linear functionals in the embedding
space at the error rate of Lemma \ref{lem2}. Below, in \S \ref{defembed}, we
define such an embedding $\Psi$, and establish the two described properties
in Lemmas \ref{expect0} and \ref{lem3}.

\subsection{Definition of the embedding}  \label{defembed} In this section, we
define an embedding $\Psi : [0,1]^d \rightarrow \mathbb{R}^p$, where
$$ 
p :=
\sum_{k = 0}^{d-1} 2^{m-k} {m \choose k} {d-1 \choose k}, 
$$ 
for some fixed integer scale $m$.  If $I$ is an interval, then let $I^-$ and
$I^+$ denote the left and right halves of $I$, respectively.  More generally,
if $R = I_1 \times \cdots \times I_d$ is a box in $[0,1]^d$, let $R_j^-$
and $R_j^+$ denote the left and right halves of $R$ with respect to the
$j$-th coordinate; more precisely, 
$$
R_j^\pm := J_1 \times \cdots \times J_d, \quad \text{where} \quad J_i := \left\{
\begin{array}{cc} I_j^\pm & \text{if } i = j \\ I_i & \text{if } i \not =
j. \end{array} \right.  
$$ 
The embedding $\Psi : [0,1]^d \rightarrow \mathbb{R}^p$ is defined using
indicator functions of dyadic boxes. If $R$ is a box, then $\chi_R(x)$ denotes
its indicator function. Each coordinate of $\Psi$ is associated with a triple
$(k,\vec{r},R)$, where $k$ is an integer scale, $\vec{r}$ is a $k+1$-dimensional
integer vector, and $R$ is a dyadic box. In the following, we define the set of
triples $\mathcal{T}$ used to define $\Psi$ by introducing helper sets
$\mathcal{R}_k$ and $\mathcal{D}_{k,\vec{r}}$.

Fix $k \in \{0,\ldots,d-1\}$, and define
$$ \mathcal{R}_{k} := \{ \vec{r} =
(r_0,\ldots,r_k) \in \mathbb{Z}^{k+1} : 1 = r_0 < r_1 < \cdots < r_{k} \le d \}.
$$ 
If
$k \in \{0,\ldots,d-1\}$ and $\vec{r} \in \mathcal{R}_K$, then we define the set
of dyadic boxes $\mathcal{D}_{k,\vec{r}}$ by $$ \mathcal{D}_{k,\vec{r}} = \big\{
R = I_1 \times \cdots \times I_d \in \mathcal{D}_d : |R| = 2^{k-m} \wedge I_j =
[0,1] \text{ if } j \not \in \{ r_0,\ldots,r_{k}\} \big\}, 
$$
where, recall, that $\mathcal{D}_d$ is the set of all dyadic boxes in
$[0,1]^d$. Finally, we define the collection of triples $\mathcal{T}$ by
\begin{equation} \label{triples}
\mathcal{T} := \{ (k,\vec{r},R) : k \in \{0,\ldots,d-1 \} \wedge \vec{r} \in
\mathcal{R}_k \wedge R \in
\mathcal{D}_{k,\vec{r}} \}.
\end{equation}
 For each $k \in \{0,\ldots,d-1\}$, there are ${
d-1 \choose k}$ possible choices of $\vec{r} \in \mathcal{R}_k$, and $2^{m-k}{m
\choose k}$ possible choices of $R \in \mathcal{D}_{k,\vec{r}}$, so we conclude
that $\# \mathcal{T} = p$. The embedding $\Psi : [0,1]^d \rightarrow
\mathbb{R}^p$ is defined by fixing an enumeration of $\mathcal{T}$; we emphasize
that all subsequent calculations are independent of the choice of enumeration.

\begin{definition} \label{Psi} Let
$(k_1,\vec{r}_1,R_1),\ldots,(k_p,\vec{r}_p,R_p)$ be a fixed enumeration of
$\mathcal{T}$. We define the embedding $\Psi : [0,1]^d \rightarrow \mathbb{R}^p$
by $$ \Psi(x) = \left( \Psi_{k_1,\vec{r}_1,R_1}(x),\ldots, \Psi_{k_p, \vec{r}_p,
R_p}(x) \right), $$ 
where 
\begin{equation}  \label{PsikrR} 
\Psi_{k,\vec{r},R}(x) :=
\left\{ \begin{array}{cl} \chi_R(x) & \text{if } k = 0 \\ 2^{-k/2}
\prod_{j=1}^{k} ( \chi_{R_{r_j}^+}(x) - \chi_{R_{r_j}^-}(x) ) & \text{if } k
\not = 0.  \end{array} \right.
\end{equation} 
\end{definition}

\subsection{Properties of the embedding} In the following, we establish two key
properties of the embedding $\Psi : [0,1]^d \rightarrow \mathbb{R}^p$. First, in
Lemma \ref{expect0} we show that distinct coordinates of the embedding are
uncorrelated.  Second, in Lemma \ref{lem3} we show that mixed H\"older functions
can be efficiently approximated by linear functionals in the embedding space.

\begin{lemma} \label{expect0} Suppose that $X$ is chosen uniformly at random
from $[0,1]^d$. Then 
$$ 
\mathbb{E} \left( \Psi_i(X) \Psi_j(X) \right) = \left\{ \begin{array}{cc} 2^{-m}
& \text{if } i = j \\ 0 & \text{if } i \not = j, \\ \end{array} \right.  
$$ 
for all $i,j \in \{1,\ldots,p\}$, where $\Psi_j(X)$ denotes the $j$-th coordinate of
$\Psi(X) \in \mathbb{R}^p$.
\end{lemma}

\begin{proof}[Proof of lemma \ref{expect0}] 
There are two cases to consider: $i=j$ and $i \not = j$.  Recall that each
coordinate of $\Psi(X)$ corresponds to a different triple $(k,\vec{r},R) \in
\mathcal{T}$.  Thus, to establish the case $i = j$ it suffices to show that 
$$
\Psi_{k,\vec{r},R}(X)^2 = 2^{-m}, 
$$ 
for all $(k,\vec{r},R) \in \mathcal{T}$.
Observe that if
$X \in R$, then $\Psi_{k,\vec{r},R}(X)^2 = 2^{-k}$, while if $X \not \in R$, then
$\Psi_{k,\vec{r},R}(X)^2 = 0$. Since the probability that $X$ is
contained in $R$ is equal to $$ \mathbb{P}(X \in R) = |R| = 2^{k-m}, $$ we
conclude that $\mathbb{E}( \Psi_{k,\vec{r},R}(X)^2 ) = 2^{-m}$. Next, to
establish the case $i
\not =j$ it suffices to show that if  
$(k,\vec{r},R) \not = (k',\vec{r'},R')$, then
$$ \mathbb{E} (\Psi_{k,\vec{r},R}(X)
\Psi_{k',\vec{r'},R'}(X) ) = 0. $$
If $R$ and $R'$ are disjoint, then the product of
$\Psi_{k,\vec{r},R}(x)$ and $\Psi_{k',\vec{r'},R'}(x)$ is equal to zero for all
$x \in [0,1]^d$.
Suppose that $R \cap R' \not = \emptyset$, and let $R = I_1 \times \cdots \times
I_d$ and $R' = I_1' \times \cdots \times I_d'$.  Under these assumptions, there
must exist an index $i_0 > 1$ such that either 
$$
|I_{i_0}| \not = |I_{i_0}'| \text{ or }   i_0  \in \{r_1,\ldots,r_k\} \setminus
\{r_1',\ldots,r_k'\} \text{ or }  i_0 \in \{r_1',\ldots,r_k'\}
\setminus \{r_1,\ldots,r_k\}.
$$ 
Indeed, otherwise $(k,\vec{r},R)$ and $(k',\vec{r'},R')$ would be identical.
Since the cases $i_0  \in \{r_1,\ldots,r_k\} \setminus \{r_1',\ldots,r_k'\}$ and
$ i_0 \in \{r_1',\ldots,r_k'\} \setminus \{r_1,\ldots,r_k\}$ are symmetric,
without loss of generality we can consider two cases.

\subsubsection*{Case 1: $|I_{i_0}| \not = |I_{i_0}'|$}  Since $I_{i_0}$ and
$I_{i_0}'$ are dyadic intervals that have nonempty intersection we may assume,
without loss of generality, that $I_{i_0} \subsetneq I_{i_0}'$. If the
$i_0$-th coordinate of a point in $R$ is moved from $I_{i_0}^+$ to $I_{i_0}^-$,
then the function $ \Psi_{k',\vec{r'},R'}$ remains constant, while the function
$\Psi_{k,\vec{r},R}$ changes sign, see Definition \ref{Psi}. We conclude that the expected value of the
product of these functions is zero.

\subsubsection*{Case 2: $|I_{i_0}| = |I_{i_0}'|$ and $i_0  \in \{r_1,\ldots,r_k\}
\setminus \{r_1',\ldots,r_k'\}$}

We argue in a similar way to Case 1. If the $i_0$-th coordinate of a point in
$R$ is moved from $I_{i_0}^+$ to $I_{i_0}^-$, then the function
$\Psi_{k',\vec{r'},R'}$ is constant because $i_0 \not \in \{r_1',\ldots,r_k'\}$,
while the function $\Psi_{k,\vec{r},R}$ changes sign because $i_0 \in
\{r_1,\ldots,r_k\}$. It follows that $\mathbb{E}( \Psi_{k,\vec{r},R}(X)
\Psi_{k',\vec{r'},R'}(X)) = 0$.  This completes the proof.
\end{proof}

Recall that $\mathcal{M}^{\alpha}([0,1]^d)$ is the space of all functions $f :
[0,1]^d \rightarrow \mathbb{R}$ that are $(c,\alpha)$-mixed H\"older for some
constant $c > 0$. In the following lemma we show that functions in
$\mathcal{M}^{\alpha}([0,1]^d)$ can be effectively approximated by linear
functionals in the embedding space $\mathbb{R}^p$.

\begin{lemma} \label{lem3}
There exists an embedding $\Phi : \mathcal{M}^{\alpha}([0,1]^d)
\rightarrow \mathbb{R}^p$ such that 
\begin{equation} \label{rate}
\left| \langle \Phi(f), \Psi(x) \rangle  - f(x) \right| \lesssim
2^{-\alpha m} m^{d-1},
\end{equation}
where the implicit constant only depends on the mixed H\"older constant $c > 0$
of the function, and dimension $d \ge 1$.
\end{lemma}

The proof strategy for this lemma is as follows. First, we define an embedding
$\Psi' : [0,1]^d \rightarrow \mathbb{R}^P$, where $P > p$. The embedding $\Psi'$
will not satisfy the uncorrelated coordinate property of Lemma \ref{expect0},
but it will be easier to define a corresponding embedding $\Phi' :
\mathcal{M}^\alpha([0,1^d) \rightarrow \mathbb{R}^P$ such that
\eqref{rate} holds. Second, we will describe an iterative compression
procedure that can be applied to the embeddings $\Psi'$ and $\Phi'$, which
preserves  \eqref{rate}, and results in the
embeddings $\Psi$ and $\Phi$.

\begin{proof}[Proof of Lemma \ref{lem3}]
Let $R_1,\ldots,R_P$ be a fixed enumeration of the $P$ dyadic boxes of measure
$2^{-m}$ in $[0,1]^d$; recall from \S \ref{preliminaries} that
$$
P = 2^m {m + d - 1 \choose d-1} \sim 2^m m^{d-1}.
$$
We define the embedding $\Psi' : [0,1]^d \rightarrow \mathbb{R}^P$
by
$$
\Psi'(x) = \left( \chi_{R_1}(x),\ldots,\chi_{R_P}(x) \right),
$$
and will define a corresponding embedding $\Phi' :
\mathcal{M}^\alpha([0,1]^d) \rightarrow \mathbb{R}^P$ such that
$$
| \langle \Phi'(f),  \Psi'(x)\rangle - f(x)| \lesssim 2^{-\alpha m} m^{d-1}.
$$
Fix $f \in \mathcal{M}^\alpha([0,1]^d)$
and $x \in [0,1]^d$, and let $c_R$ denote the center of the dyadic box $R$; we
can write the result of Lemma \ref{lem2} as
\begin{equation} \label{smolyak2}
\left| f(x) - \sum_{k=0}^{\min\{d-1,m\}} (-1)^k {d -1 \choose k} \sum_{R \in
\mathcal{D}_d :  x \in R \wedge  |R|=2^{k-m}} f(c_R) \right| \lesssim 2^{-\alpha
m} m^{d-1} .
\end{equation}
Let $R'$ be a fixed dyadic box in $[0,1]^d$ of measure $|R'| = 2^{k-m}$. A
counting argument shows that the number of dyadic boxes of measure $2^{-m}$ that
are contain in $R'$ and contained a fixed point $x \in R'$ is
\begin{equation} \label{count3}
\# \{ R \in \mathcal{D}_d : |R| = 2^{-m} \wedge R \subseteq R' \wedge x \in R \}
= {k + d - 1 \choose d - 1}.
\end{equation}
Motivated by \eqref{count3}, we define the embedding $\Phi' :
\mathcal{M}^\alpha([0,1]^d) \rightarrow \mathbb{R}^P$ entry-wise by
$$
\Phi'_j(f) := \sum_{k=0}^{\min\{d-1,m\}} (-1)^k  \frac{{d-1 \choose k}}{{k + d
- 1 \choose d - 1}} \sum_{R \in \mathcal{D}_d :  R_j  \subseteq R \wedge |R| =
2^{k-m}} f(c_R),
$$
for $j \in \{1,\ldots,P\}$. By construction we have
$$
\langle \Phi'(f) , \Psi'(x) \rangle = 
\sum_{k=0}^{\min\{d-1,m\}} (-1)^k {d -1 \choose k} \sum_{R \in
\mathcal{D}_d :  x \in R \wedge  |R|=2^{k-m}} f(c_R),
$$
and by Lemma \ref{deterministic} we conclude that
\begin{equation} \label{rateb}
\left|\langle \Phi'(f) , \Psi'(x)\rangle  - f(x) \right| \lesssim 2^{-\alpha m}
m^{d-1},
\end{equation}
which completes the first step of the proof. 

Next, we iteratively compress $\Phi'$ and $\Psi'$. Set $\Phi^{(0)}
:= \Phi'$ and $\Psi^{(0)} := \Psi'$; we will define $\Phi^{(i)} :
\mathcal{M}^\alpha([0,1]^d) \rightarrow \mathbb{R}^{P-i}$ and $\Psi^{(i)} :
[0,1]^d \rightarrow \mathbb{R}^{P-i}$ for $i = 0,\ldots,P-p$ such that
$$
\langle \Phi^{(i)}(f), \Psi^{(i)}(x) \rangle = \langle \Phi^{(i+1)}(f),
\Psi^{(i+1)}(x) \rangle.
$$
Moreover, these
embeddings will be defined such that $\Psi^{(P-p)} = \Psi$ is the embedding
defined in Definition \ref{Psi} and $\Phi =
\Phi^{(P-p)}$ is the desired corresponding embedding.

Each of the coordinates of the embeddings $\Psi^{(i)}$ will be associated with a
triple $(k,\vec{r},R)$, where $k \in \{0,\ldots,d-1\}$, $\vec{r} \in
\mathbb{Z}^{k+1}$, and $R \in \mathcal{D}^d \wedge |R| = 2^{m-k}$. Consistent with 
\eqref{PsikrR} we define
\begin{equation} \label{PsikrR2}
\Psi^{(i)}_{k,\vec{r},R}(x) = \left\{ \begin{array}{cc}
\chi_R(x) & \text{if }  k = 0 \\
2^{-k/2} \prod_{j=1}^k (\chi_{R_{r_j}^+}(x) - \chi_{R_{r_j}^-}(x))
& \text{if } k \not = 0.
\end{array} \right.
\end{equation}
Let $\mathcal{T}^{(i)}$ denote the set of triples that index the coordinates of
$\Psi^{(i)}$. Initially, 
$$
\mathcal{T}^{(0)} = \{(0,(1),R) : R \in \mathcal{D}_d \wedge |R| = 2^{-m}
\}.
$$
If $R = I_1 \times \cdots \times I_d$ is a box, then let $\pi_j R = I_j$ denote
the projection of $R$ on the $j$-th coordinate. Fix a coordinate $l \in
\{2,\ldots,d\}$ and a scale $s \in \{1,\ldots,d\}$. Suppose that
$\mathcal{T}^{(i)}$ is given, and suppose that we have performed the following
procedure for coordinates strictly greater than $l$ at all scales, and at scales
greater than $s$ for the coordinate $l$. Choose $(k,\vec{r},R), (k,\vec{r},R')
\in \mathcal{T}^{(i)}$ such that
$$
R = \bar{R}^+_l, \quad R' = \bar{R}_l^-, \quad \text{and} \quad
|\pi_l \bar{R}_l^+| = |\pi_l \bar{R}_l^-| = 2^{-s},
$$
for some dyadic box $\bar{R}$.  Set
$$
\mathcal{T}^{(i+1)} = \left( \mathcal{T}^{(i)} \setminus
\{(k,\vec{r},\bar{R}_l^+),(k,\vec{r},\bar{R}_l^-)\} \right) \cup
\{(k+1,\vec{r}_l,\bar{R})\},
$$
where $\vec{r}_l = (1,l,r_1,\ldots,r_k) \in \mathbb{Z}^{k+2}$. Let the embedding
$\Psi^{(i+1)} : [0,1]^d \rightarrow \mathbb{R}^{P-i-1}$ be defined by
\eqref{PsikrR2}. Observe that
$$
\Psi^{(i+1)}_{k+1,\vec{r}_l,\bar{R}}(x) =
\frac{\Psi^{(i)}_{k,\vec{r},\bar{R}_l^+}(x) -
\Psi^{(i)}_{k,\vec{r},\bar{R}_l^-}(x)}{\sqrt{2}}. 
$$
We define the new coordinate of the embedding $\Phi^{(i+1)}$ associated
with the triple $(k+1,\vec{r}_l,R)$ by
$$
\Phi^{(i+1)}_{k+1,\vec{r}_l,\bar{R}}(f) =
\frac{\Phi^{(i)}_{k,\vec{r},\bar{R}_l^+} (f) -
\Phi^{(i)}_{k,\vec{r},\bar{R}_l^-}(f)}{\sqrt{2}}.
$$
Additionally, we modify the coordinates indexed by $(k,\vec{r},\bar{R}_1^+)$ and
$(k,\vec{r},\bar{R}_1^-)$  by
$$
\Phi^{(i+1)}_{k,\vec{r},\bar{R}_1^\pm}(f) =
\Phi^{(i)}_{k,\vec{r},\bar{R}_1^\pm}(f) + \frac{
\Phi^{(i)}_{k,\vec{r},\bar{R}_l^+}(f) + \Phi^{(i)}_{k,\vec{r},\bar{R}_l^-}}{2}.
$$
Otherwise, we set 
$$
\Phi^{(i+1)}_{k,\vec{r},R} = \Phi^{(i)}_{k,\vec{r},R},
$$
for $(k,\vec{r},R) \in \mathcal{T}^{(i)} \setminus
\{(k,\vec{r},\bar{R}^+_1),(k,\vec{r},\bar{R}_1^-)\}$. 
The basic idea is to replace two distinct values with their difference and
average distributed over several blocks. This transformation reduces the number
of coordinates by $1$ while preserving the inner product between $\Psi^{(i)}$
and $\Phi^{(i)}$, see the illustration in Figure \ref{fig05}.

\begin{figure}[h!]
\includegraphics[width=.9\textwidth]{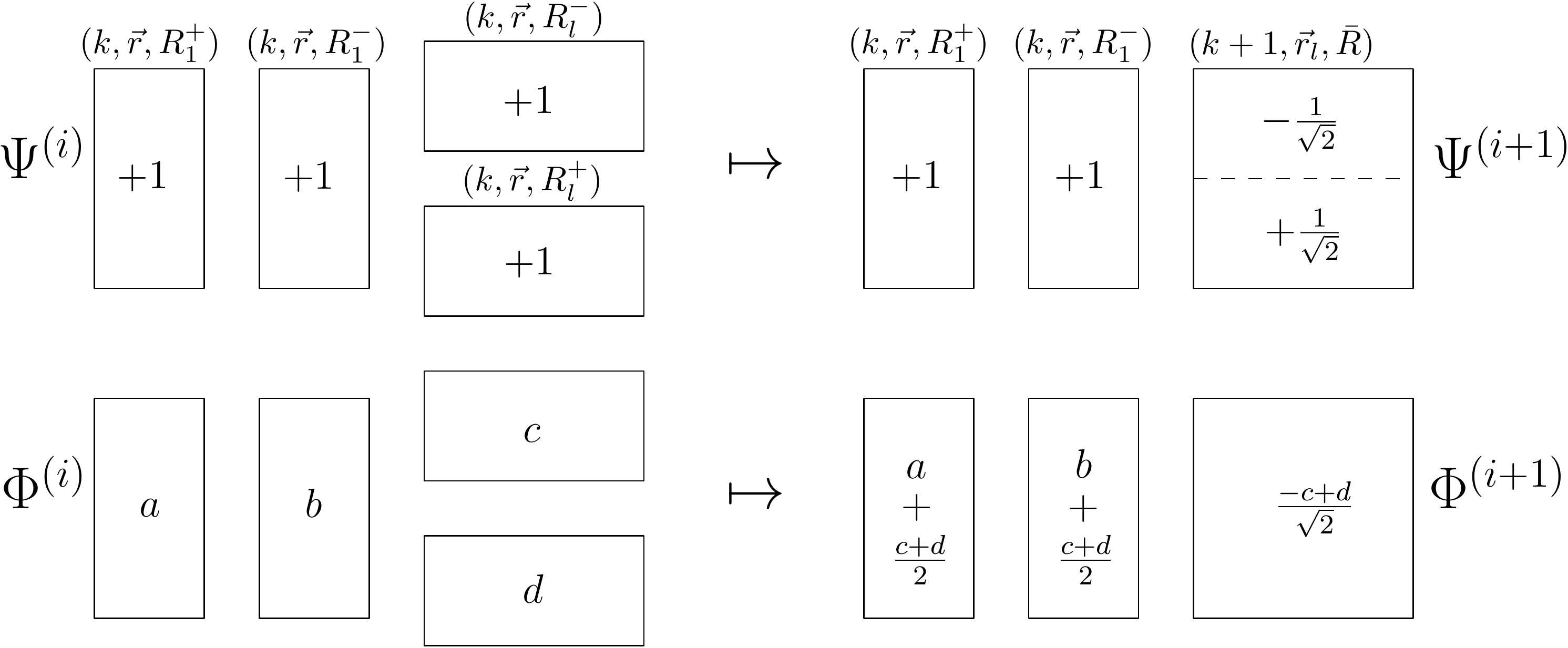}
\caption{The transformation from $\Psi^{(i)}$ and $\Phi^{(i)}$ to $\Psi^{(i+1)}$
and $\Phi^{(i+1)}$ in the simplest case where $\Psi^{(i)}$ is initially
constant on $R^\pm_1$.} \label{fig05}
\end{figure}

The procedure illustrated in Figure \ref{fig05} is similar to the
calculation of the Haar wavelet coefficients of a function.  By the
definition of $\Phi^{(i+1)}$ and $\Psi^{(i+1)}$ it is straightforward to verify
that 
$$
\langle \Phi^{(i)}(f) , \Psi^{(i)}(x) \rangle = \langle \Phi^{(i+1)}(f),
\Psi^{(i+1)}(x) \rangle.
$$
Since set of triples $\mathcal{T}$ was defined in \eqref{triples} to correspond
to the result of applying the above described compression procedure at all
scales and all coordinates $j > 1$, by iterating the described procedure
for $i = 0,\ldots,P - p$ we conclude that
$$
\langle \Phi(f), \Psi(x) \rangle = \langle \Psi'(f) , \Psi'(x) \rangle,
$$
and by \eqref{rateb} the proof is complete. 
\end{proof}

\section{Proof of main result} \label{proofofmain}
\subsection{Summary}
So far we have defined the embedding $\Psi : [0,1]^d \rightarrow \mathbb{R}^p$,
see Definition \ref{Psi}, and established that $\Psi$ has two key
properties: first, the coordinates of $\Psi$ are uncorrelated, see Lemma
\ref{expect0}, and second, mixed H\"older functions can be effectively
approximated by linear functionals acting on the $\Psi$ coordinates of a point,
see Lemma \ref{lem3}.  To complete the proof of Theorem \ref{thm1}, we will use
the embedding $\Psi$ together with the randomized Kaczmarz algorithm of Lemma
\ref{randomkaczmarz}.

\subsection{Proof of Theorem \ref{thm1}}
\begin{proof}[Proof of Theorem \ref{thm1}]
Let $\Psi : [0,1]^d \rightarrow \mathbb{R}^p$ be the embedding defined in
Definition \ref{Psi}, and let $x_1,\ldots,x_{2^{d m}}$ be a sequence of points
that contains exactly one point in each dyadic box that is contained in
$[0,1]^d$ and has dimensions $2^{-m} \times \cdots \times 2^{-m}$. Let $A$ be
the $2^{d m} \times p$ matrix whose $j$-th row is $\Psi(x_j)$. By definition,
the vector $\Psi(x)$ has magnitude
$$
\| \Psi(x) \|_2 = \sqrt{ \sum_{k=0}^{d-1} 2^{-k} {m \choose k} {d - 1 \choose k}},
$$
for all $x \in [0,1]^d$. Thus, all of the rows of $A$ have equal
magnitude. Furthermore, since Lemma \ref{expect0} implies that all of the
singular values of $A$ are equal to $2^{m (d-1)/2}$, it follows that the
condition number $\kappa^2$ of $A$ is equal to
$$
\kappa^2 = \sum_{j=1}^p \sigma_j^2 /\sigma_p^2 =  1/p.
$$

Next, we construct a consistent linear system of equations using the matrix $A$.
By Lemma \ref{lem3}, there exists a vector $w \in \mathbb{R}^p$ that
depends on $f$ such that
$$
\left| \langle w , \Psi(x) \rangle  - f(x) \right| \lesssim 2^{-\alpha m} m^{d-1}.
$$
We define the function $\bar{f} : [0,1]^d \rightarrow \mathbb{R}$ by
$$
\bar{f}(x) := \langle w, \Psi(x) \rangle.
$$
Let $b$ be the $2^{d m}$-dimensional vector whose $j$-th entry is $\bar{f}(x)$.
Consider the consistent linear system of equations
$$
A w = b.
$$
Recall that the randomized Kaczmarz algorithm described in Lemma
\ref{randomkaczmarz} involves sampling rows from $A$ uniformly at random. By
construction, sampling points uniformly at random from $[0,1]^d$ and applying
$\Psi$ is equivalent to sampling rows uniformly at random from $A$. Therefore,
by Lemma \ref{randomkaczmarz} we conclude that if $X_1,X_2,\ldots$ are chosen
independently and uniformly at random from $[0,1]^d$, $w_0 \in \mathbb{R}^p$ is
an initial vector, and \begin{equation} \label{itr} w_n := w_{n-1} +
\frac{\bar{f}(x) -\langle
\Psi(X_n), w_{n-1} \rangle
}{\|\Psi(X_n)\|_2^2}
\Psi(X_n),
\end{equation}
then
\begin{equation} \label{estw}
\mathbb{E} \|w_n - w\|^2_2 \le \left(1 - \textstyle{\frac{1}{p}}
\right)^n \|w_0 - w\|_2^2.
\end{equation}
Next, we quantify the error produced by replacing
$\bar{f}$ with ${f}$ in \eqref{itr}. For an initial vector $w_0^* \in
\mathbb{R}^p$, set
\begin{equation} \label{wstar}
w_n^* = w_{n-1}^* + \frac{f(X_n) - \langle \Psi(X_n), w_{n-1}^*
\rangle}{\|\Psi(X_n)\|_2^2} \Psi(X_n).
\end{equation}
Let $\epsilon_n := f(X_n) - \bar{f}(X_n)$, and define
$$
w_n := w_{n-1} + \frac{\bar{f}(X_n) - \langle \Psi(X_n), w_{n-1}
\rangle}{\|\Psi(X_n)\|_2^2} \Psi(X_n),
$$
and
$$
e_n := e_{n-1} + \frac{\epsilon_n - \langle \Psi(X_n), e_{n-1}
\rangle}{\|\Psi(X_n)\|^2_2} \Psi(X_n).
$$
If $e_0$ is the all zero vector, and $w_0 = w_0^*$, then it follows by induction
that 
$$
w_n^* = w_n + e_n,
$$
for all $n$.  Thus, by the triangle inequality,
$$
\|w_n^* - w\|_2 \le \|w_n -w \|_2 + \|e_n\|_2.
$$
First, we estimate $\|e_n\|_2$.  By orthogonality we have
\begin{equation} \label{ortho}
\|e_n\|^2_2 = \left\| e_{n - 1} - \frac{\langle \Psi(X_n), e_{n-1}
\rangle}{\|\Psi(X_n)\|_2^2} \Psi(X_n) \right\|^2_2
+ \left\| \frac{\epsilon_n }{\|\Psi(X_n)\|^2} \Psi(X_n) \right\|^2_2.
\end{equation}
The first term on the right hand side of \eqref{ortho} is the projection of
$e_{n-1}$ on the subspace orthogonal to $\Psi(X_n)$ so we conclude
$$
\|e_n\|^2_2 \le \|e_{n-1}\|^2_2 + \frac{\epsilon_n^2}{\|\Psi(X_n)\|^2}_2.
$$
By induction, it follows that
$$
\|e_n\|^2_2 \le \sum_{j=1}^n \frac{\epsilon_j^2}{\|\Psi(X_j)\|^2} \lesssim
n 2^{-2 \alpha m} m^{d-1},
$$
where the final step uses the estimates $|\epsilon_n| \lesssim 2^{-\alpha m}
m^{d-1}$ and $\|\Psi(X_n)\|^2_2 \sim m^{d-1}$. Next, we develop a high
probability bound on $\|w_n - w\|_2$. By \eqref{estw} we have
$$
\mathbb{E} \|w_n - w\|^2_2 \le \left(1 - \textstyle{\frac{1}{p}}
\right)^n \|w_0 - w\|_2^2.
$$
By the possibility of considering $f(x) - f(X_1)$ instead of $f(x)$, we may
assume that $|f| \lesssim 1$ on $[0,1]^d$, where the implicit constant only
depends on the mixed H\"older constant $c > 0$, and the dimension $d \ge 1$. It
follows that $\|w\|_2^2 \lesssim p$. Thus, if $n \ge c_1 p \log(2^m)$, then
$$
\mathbb{E} \|w_n - w\|^2_2 \lesssim p 2^{-c_1 m}.
$$
By Markov's inequality
$$
\mathbb{P} \left( \|w_n - w\|_2^2 \ge p 2^{-2 \alpha m} m^{d-1} \right) \le
\frac{\mathbb{E} \|w_n - w \|_2^2}{p 2^{-2 \alpha m } m^{d-1}} \lesssim
\frac{2^{-(c_1 - 2 \alpha)m}}{m^{d-1}}. 
$$
Therefore, when $m$ is large enough in terms of the mixed H\"older constant $c >
0$ and the dimension $d \ge 1$, we have
$$
\mathbb{P}\left( \|w_n - w\|_2 \ge \sqrt{p} 2^{-\alpha m} m^\frac{d-1}{2}
\right) < 2^{-(c_1 -2 \alpha)m}.
$$
Thus, when $n = \lceil c_1 p \log(2^m) \rceil$, we have
$$
\|w_n^* - w\|_2 \lesssim \sqrt{p \log(2^m)} 2^{-\alpha m} m^{\frac{d-1}{2}}
\lesssim 2^{(1/2 -\alpha)m} m^{d - 1/2},
$$
with probability at least $1 - 2^{-(c_1 - 2 \alpha)m}$, where the implicit
constant depends on the mixed H\"older constant $c > 0$, the dimension $d \ge
1$, and the constant $c_1 > 0$. Recall that by Lemma \ref{expect0} all of the
singular values of the matrix $A$ are $2^{m(d-1)/2}$. It follows that
$$
\|A w_n^* - A w\|_2 \lesssim 2^{(d/2 - \alpha)m} m^{d-1/2}.
$$
We define the function $\tilde{f} : [0,1]^2 \rightarrow \mathbb{R}$ by
$$
\tilde{f}(x) : = \langle w_n^*, \Psi(x) \rangle.
$$
Recall that $\bar{f}(x) := \langle \Psi(x), w \rangle$. Thus, we have
$$
\|\tilde{f} - \bar{f}\|_{L^2} = \sqrt{ \int_{[0,1]^d} (\tilde{f}(x) -
\bar{f}(x))^2 dx} =  2^{-d m/2} \| A w_n^* - A w\|_2 \lesssim 2^{-\alpha m} m^{d -
1/2}.
$$
Since $\|f - \bar{f} \|_{L^2} \lesssim 2^{-\alpha m} m^{d-1}$, it follows from
the triangle inequality that
$$
\| f - \tilde{f} \|_{L^2} \lesssim 2^{-\alpha m} m^{d-1/2},
$$
with probability at least $1 - 2^{-(c_1 - 2\alpha)m}$, where the implicit
constant only depends on the constants $c > 0$ and $c_1 > 0$, and the dimension
$d \ge 1$. Setting $\varepsilon = 2^{-m}$ completes the proof.
\end{proof}

\subsection{Proof of computational cost} \label{compcostproof} 
\begin{proof}[Proof of Remark \ref{compcost}]
Fix $\varepsilon = 2^{-m}$. We claim that $\Psi(x)$ can be represented by a 
sparse vector with 
$$
\sum_{k=0}^{d-1} {m \choose k} {d - 1 \choose k} \sim (\log
\textstyle{\frac{1}{\varepsilon}})^{d-1}
$$
nonzero entries. Indeed, recall that each of the coordinates of $\Psi(x)$ is
associated with a triple $(k,\vec{r},R)$, where $k$ is an integer in
$\{0,\ldots,d-1\}$, $\vec{r}$ is an integer vector in $\mathcal{R}_k$, and $R$
is a dyadic box in $\mathcal{D}_{k,r}$. By Definition \ref{Psi}, we have
$$
\Psi_{k,\vec{r},R}(x) \not =0  \quad
\iff \quad x \in R.
$$ 
For each $k \in
\{0,\ldots,d-1\}$ there are ${d -1 \choose k}$ distinct vectors in
$\mathcal{R}_k$. The set $\mathcal{D}_{k,\vec{r}}$ contains $2^{m -k } { m
\choose k}$ dyadic rectangles of measure $2^{k-m}$, which can be divided into ${
m \choose k}$ sets such that the dyadic boxes in each set are disjoint, have the
same dimension, and cover $[0,1]^d$. Thus, each of these sets has one dyadic box
that contains $x$, which can be determined in $\sim 1$ operations. It follows
that we can construct a sparse vector that represents $\Psi(x)$ in $\sim (\log
\frac{1}{\varepsilon} )^{d-1}$ operations. A detailed description of the
construction of the sparse embedding vector $\Psi(x)$ for a point $x \in
[0,1]^d$ for the case $d = 3$ is described in Algorithm \ref{algo3}.

Thus, given a set of $n$ points $X_1,\ldots,X_n$, we can compute the
corresponding sparse embedding vectors $\Psi(X_1),\ldots,\Psi(X_n)$ in $\sim n (
\log \frac{1}{\varepsilon})^{d-1}$ operations. Given these sparse embedding
vectors, computing $w_n^*$ requires $n$ sparse vector inner
products and $n$ sparse vector additions for a total of $\sim n (\log
\frac{1}{\varepsilon})^{d-1}$ operations, see \eqref{wstar}. Thus, when $n \sim
\frac{1}{\varepsilon} ( \log \frac{1}{\varepsilon} )^d$ the total computational
cost of constructing $w_n^*$ is $\sim \frac{1}{\varepsilon} (\log
\frac{1}{\varepsilon})^{2d -1}$ operations.

After computing $w_n^*$, we can construct $\Psi(x)$ and then evaluate
$\tilde{f}(x) = \langle w_n^*, \Psi(x) \rangle$ for any $x \in [0,1]^d$ in $\sim
(\log \frac{1}{\varepsilon} )^{d-1}$ operations. Furthermore, given $w_n^*$ we
can compute the integral of $\tilde{f}$ in $\sim \frac{1}{\varepsilon}$
operations. Indeed, if $x_1,\ldots,x_{2^{dm}}$ is a sequence of points that
contains exactly one point in each dyadic box with equal side lengths and
measure $2^{-dm}$, then
$$
\int_{[0,1]^d} \tilde{f}(x) dx = \varepsilon^3 \sum_{j=1}^{2^{dm}} \langle w_n^*, \Psi(x_j) \rangle =
\left\langle w_n^*, \varepsilon^3 \sum_{j=1}^{2^{dm}} \Psi(x_j) \right\rangle.
$$
We claim that $\sum_{j=1}^{2^{d m}} \Psi(x_j)$ is the vector whose first
$\frac{1}{\varepsilon}$ entries have value $\frac{1}{\varepsilon^2}$, and whose
remaining entries are equal to zero. Indeed, by Definition \ref{Psi} the first
$\frac{1}{\varepsilon}$ entries of $\Psi$ are equal to $1$, and the remaining
coordinates of $\Psi$ each have a fixed magnitude and are positive and negative
an equal number of times. It follows that we can compute the integral of
$\tilde{f}$ in $\sim \frac{1}{\varepsilon}$ operations.
\end{proof}

\section{Algorithm details and numerical example} \label{algoandex}

\subsection{Embedding in three dimensions} \label{algo3}
In this section, we give a detailed description of the algorithm for
approximating mixed H\"older functions in dimension $d=3$.
In particular, we describe the construction of the embedding $\Psi$ in detail. 
For a fixed scale $m$, the embedding $\Psi$ associates each point in $[0,1]^3$
with a $p$-dimensional vector, where
$$
p =
\sum_{k = 0}^{2} {2 \choose k}   {m \choose k} 2^{m-k}= 2^m + 2m 2^{m-1} +
\frac{m(m-1)}{2} 2^{m-2},
$$
see Definition \ref{Psi}. In the general $d$-dimensional case, each coordinate
of $\Psi$ corresponds to a triple $(k,\vec{r},R)$ in the set 
$$
\mathcal{T} = \left\{ (k,\vec{r},R) : k \in \{0,\ldots,d-1\}, \vec{r} \in
\mathcal{R}_k , R \in \mathcal{D}_{k,\vec{R}} \right\},
$$
where
$$
\mathcal{R}_k = \left\{ \vec{r} = (r_0,\ldots,r_k) \in \mathbb{Z}^{k+1} : 1 =
r_0 < r_1 < \cdots < r_k \le d \right\},
$$
and
$$
\mathcal{D}_{k,\vec{R}} = \left\{ R \in \mathcal{D}_d : |R|^{k-m} \wedge I_j =
[0,1] \text{ if } j \not \in \{r_0,\ldots,r_k\}  \right\}.
$$
When $d=3$, we have $k \in \{0,1,2\}$, and the sets $\mathcal{R}_k$ for $k \in
\{0,1,2\}$ are
$$
\mathcal{R}_0 = \{(1)\}, \quad \mathcal{R}_1 = \{(1,2),(1,3)\}, \quad \text{and}
\quad \mathcal{R}_2 = \{(1,2,3)\}.
$$
Let $I_j^k := [(j-1) 2^{-k} , j 2^{-k})$. The sets $\mathcal{D}_{k,\vec{r}}$ for
$k \in \{0,1,2\}$ and $\vec{r} \in \mathcal{R}_k$ are
$$
\mathcal{D}_{0,(1)} = \left\{ I_j^m \times I_1^0 \times I_1^0 :
j = 1,\ldots,2^m \right\},
$$
\begin{multline*}
\mathcal{D}_{1,(1,2)} = \left\{ I_{j_1}^{k_1} \times I_{j_2}^{k_2} \times I_1^0
: k_1 = (m-1)-k_2, k_2 = 0,\ldots,m-1 ,  \right. \\ \left.  j_1 =
1,\ldots,2^{k_1}, j_2 = 1,\ldots,2^{k_2} \right\} ,
\end{multline*}
\begin{multline*}
\mathcal{D}_{1,(1,3)} = \left\{ I_{j_1}^{k_1} \times I_1^0 \times  I_{j_3}^{k_3}
: k_1 = (m-1) - k_3, k_3 = 0,\ldots,m-1 ,  \right. \\ \left. j_1 = 
1,\ldots,2^{k_1} , j_3 = 1,\ldots,2^{k_3} \right\} ,
\end{multline*}
and 
\begin{multline*}
\mathcal{D}_{2,(1,2,3)} = \left\{ I_{j_1}^{k_1} \times I_{j_2}^{k_2} \times
I_{j_3}^{k_3}:  k_1 = (m - 2) -k_2 - k_3, k_2 \in \{0,\ldots,m-2\},   \right. \\
\left.  k_3 \in \{0,\ldots,(m-2) -k_1\}, j_1 = 1,\ldots,2^{k_1},j_2 =
1,\ldots,2^{k_2}, j_3 = 1,\ldots,2^{k_3} \right\}.
\end{multline*}
Observe that $\mathcal{D}_{0,(1)}$ has $2^m$ elements, $\mathcal{D}_{1,(1,2)}$
and $\mathcal{D}_{1,(1,3)}$ each have $m 2^m$ elements, and
$\mathcal{D}_{2,(1,2,3)}$ has $(m(m-1)/2)2^{m-2}$, which accounts for all
$$
p = 2^m + 2m 2^{m-2} + \frac{m(m-1)}{2} 2^{m-2}
$$
elements of $\mathcal{T}$.  Using the explicit form of the sets
$\mathcal{D}_{k,\vec{r}}$ described above and Definition \ref{Psi} it is
straightforward to construct $\Psi$, see
Algorithm \ref{algo}.

\begin{algorithm}[h!]
\textbf{input:} vector $x = (x_1,x_2,x_3) \in [0,1]^3$, 
integer $m \ge 2$  \hfill\,\\
\textbf{output:} 
sparse vector $\Psi(x)  = (\Psi_1(x),\ldots,\Psi_p(x)) \in \mathbb{R}^p$ \hfill
\, 
\begin{enumerate}[1:]
\item \textit{initialization}
\item $\Psi_1(x) = 0,\ldots,\Psi_p(x) = 0$
\item
\item \textit{entries corresponding to $\mathcal{D}_{0,(1)}$}
\item $i = \lfloor 2^m x_1 \rfloor + 1$
\item $\Psi_i(x) = 1$
\item
\item \textit{entries corresponding to $\mathcal{D}_{(1,(1,2)}$}
\item \textbf{for} $k_2 = 0,\ldots,m-1$
\item \quad $k_1 = m-1-k_1$
\item \quad $j_1 = \lfloor 2^{k_1} x_1 \rfloor$
\item \quad $j_2 = \lfloor 2^{k_2} x_2 \rfloor$
\item \quad $s_2 = 2(\lfloor 2^{k_2+1} x_2 \rfloor \mod 2)-1$
\item \quad $i = 2^m + k_2 2^{m-1} + j_1 2^{k_2} + j_2+1$
\item \quad $\Psi_{i}(x) = s_2/\sqrt{2}$
\item
\item \textit{entries corresponding to $\mathcal{D}_{(1,(1,3)}$}
\item \textbf{for} $k_3 = 0,\ldots,m-1$
\item \quad $k_1 = m-1-k_3$
\item \quad $j_1 = \lfloor 2^{k_1} x_1 \rfloor$
\item \quad $j_3 = \lfloor 2^{k_3} x_3 \rfloor$
\item \quad $s_3 = 2(\lfloor 2^{k_3+1} x_3 \rfloor \mod 2)-1$
\item \quad $i = 2^m + m 2^{m-1} + k_3 2^{m-1} + j_1 2^{k_3} + j_3+1$
\item \quad $\Psi_{i}(x) = s_3/\sqrt{2}$
\item
\item \textit{entries corresponding to $\mathcal{D}_{2,(1,2,3)}$}.
\item \textbf{for} $k_2 = 0,\ldots,m-2$
\item \quad \textbf{for} $k_3 = 0,\ldots,m-2 - k_1$
\item \quad \quad $k_1 = m-2-k_2 - k_3$
\item \quad \quad $j_1 = \lfloor 2^{k_1} x_1 \rfloor$
\item \quad \quad $j_2 = \lfloor 2^{k_2} x_2 \rfloor$
\item \quad \quad $j_3 = \lfloor 2^{k_3} x_3 \rfloor$
\item \quad \quad $s_2 = 2(\lfloor 2^{k_2+1} x_2 \rfloor \mod 2) - 1$
\item \quad \quad $s_3 = 2(\lfloor 2^{k_3+1} x_3 \rfloor \mod 2) - 1$
\item \quad \quad $i = (m+1)2^m  +   k_2 (2 m -1 - k_2) 2^{m-3} + k_3
2^{m-2} + j_1 2^{k_2 + k_3} + j_2 2^{k_3} + j_3 + 1$
\item \quad \quad $\Psi_{i}(x) = s_1 s_2 / 2 $
\item
\item \textbf{return} $\Psi(x)$
\end{enumerate}
\caption{Constructing the embedding $\Psi : [0,1]^3 \rightarrow \mathbb{R}^p$}
\label{algo}
\end{algorithm}

\subsection{Applying randomized Kaczmarz} \label{algorand}
Suppose that $f : [0,1]^3 \rightarrow \mathbb{R}^3$ is a $(c,\alpha)$-mixed
H\"older function that is sampled at $n$ points $X_1,\ldots,X_n$
chosen independently and uniformly at random from $[0,1]^3$. Fix $\varepsilon =
2^{-m}$ for a positive scale $m$.  Using Algorithm \ref{algo} we can construct
the sequence of sparse embedding vectors $\Psi(X_1),\ldots,\Psi(X_n)$
corresponding to these points in $\sim n (\log \frac{1}{\varepsilon})^2$
operations. Next, we use the randomized Kaczmarz procedure described in Lemma
\ref{randomkaczmarz}. Given an intial guess $w_0^* \in \mathbb{R}^p$ we define
$$
w_n^* := w_{n-1}^* +  \frac{f(X_n) -\langle \Psi(X_n), w_{n-1}^* \rangle
}{\|\Psi(X_n)\|_2^2} \Psi(X_n),
$$
for $n = 1,2,\ldots$. If $n = \lceil c_1 \frac{1}{\varepsilon} (\log
\frac{1}{\varepsilon})^3 \rceil$, and 
$$
\tilde{f}(x) := \langle w_n^*, \Psi(x) \rangle,
$$
then it follows from Theorem \ref{thm1}, that
\begin{equation} \label{accb}
\|f - \tilde{f}\|_{L^2} \lesssim \varepsilon^\alpha \left(\log
\textstyle{\frac{1}{\varepsilon}} \right)^{5/2},
\end{equation}
with probability at least $1 - \varepsilon^{c_1 - 2 \alpha}$.

\subsection{Numerical Example} \label{numericalexample}
In this section, we present a numerical example of approximating a mixed
H\"older function on $[0,1]^3$ using the method detailed in \S \ref{algo3} and
\S \ref{algorand}. We start by using fractional Brownian motion to construct an
example of a function $f : [0,1]^3 \rightarrow \mathbb{R}$ that is
$\alpha$-mixed H\"older for $\alpha = .79$.  For a given parameter $0 < h < 1$,
fractional Brownian motion $B_h(t)$ is a continuous time stochastic process
that starts at zero, has expectation zero for all $t$, and has covariance
function
$$
\mathbb{E} \left( B_h(t) B_h(s) \right) = \frac{1}{2} \left( t^{2 h} + s^{2 h} -
(t-s)^{2h}\right) \quad \text{for} \quad t > s.
$$
When $h=1/2$, the covariance function $\mathbb{E} ( B_{h}(t) B_{h}(s) ) = s$,
and it follows that $B_h(t)$ is standard Brownian motion. Just as standard
Brownian motion is almost surely $\alpha$-H\"older continuous for any $\alpha <
1/2$, fractional Brownian motion parameterized by $0 < h < 1$ is almost surely
$\alpha$-H\"older continuous for any $\alpha < h$. To create a numerical
example, we fix $h = .8$ and simulate three fractional Brownian motions
$B_{.8}^{(1)}$, $B_{.8}^{(2)}$, and $B_{.8}^{(3)}$, see Figure
\ref{brownian}.
\begin{figure}[h!]
\begin{tabular}{ccc}
\includegraphics[width=.3\textwidth]{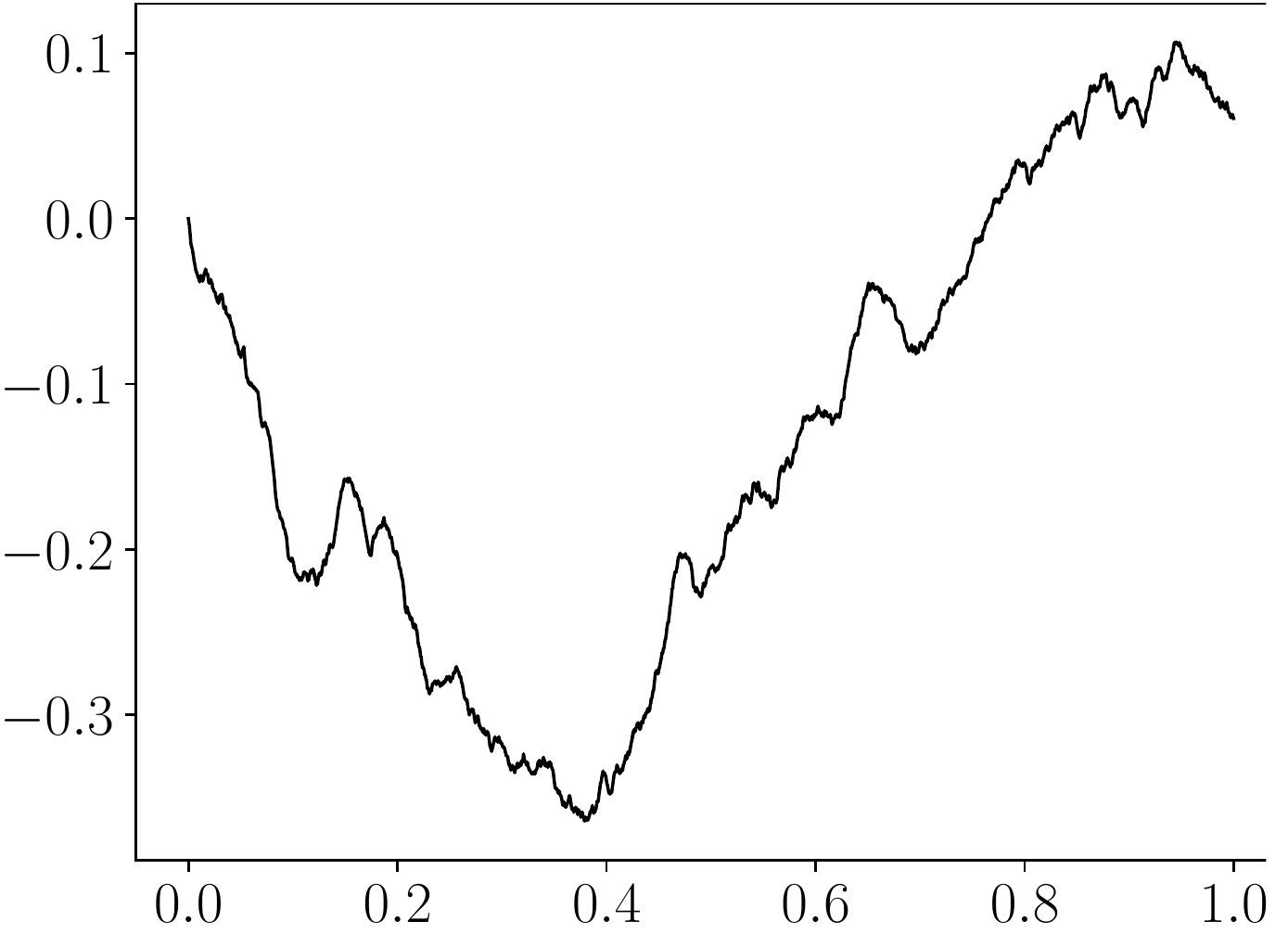} &
\includegraphics[width=.3\textwidth]{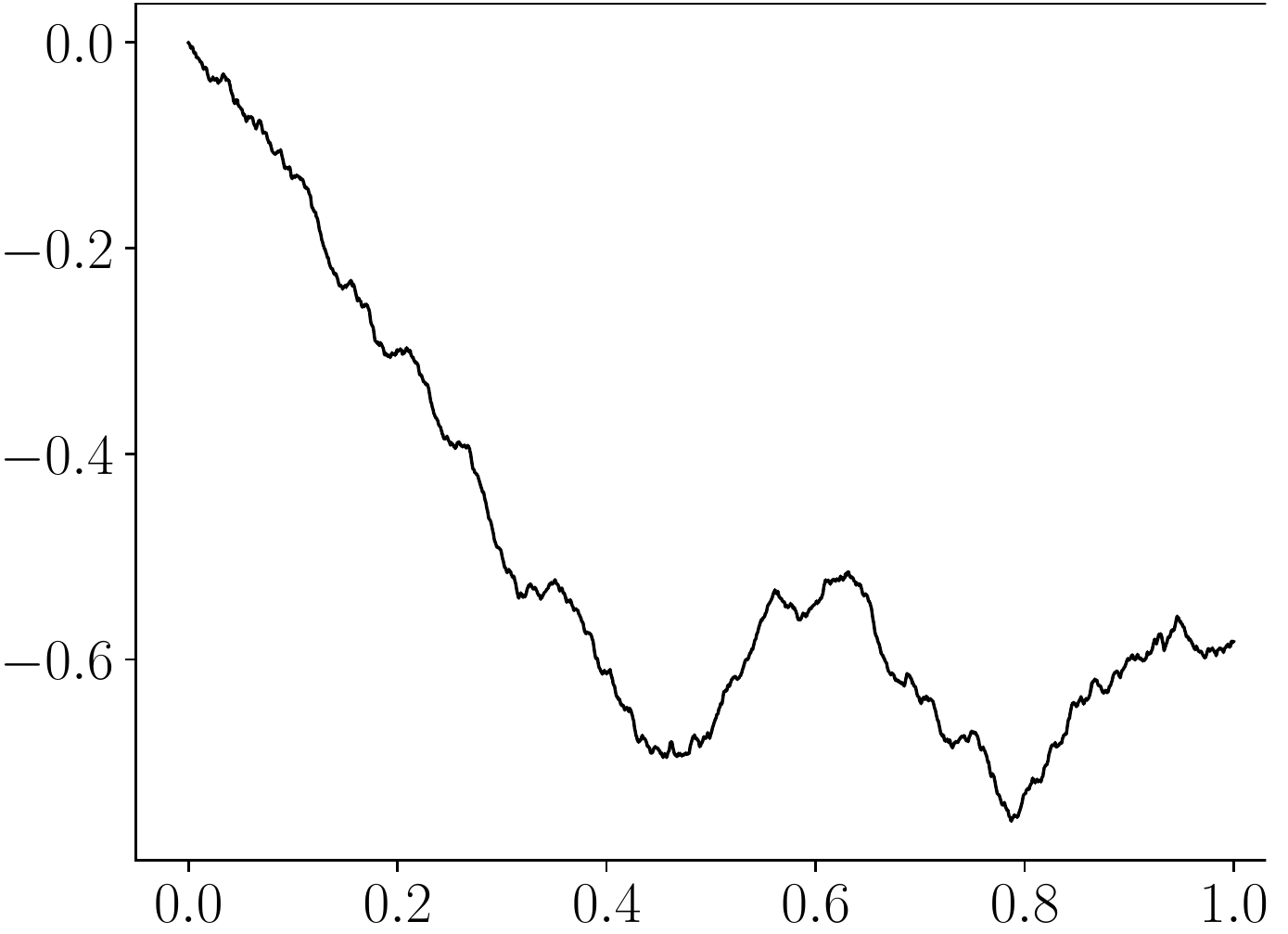} &
\includegraphics[width=.3\textwidth]{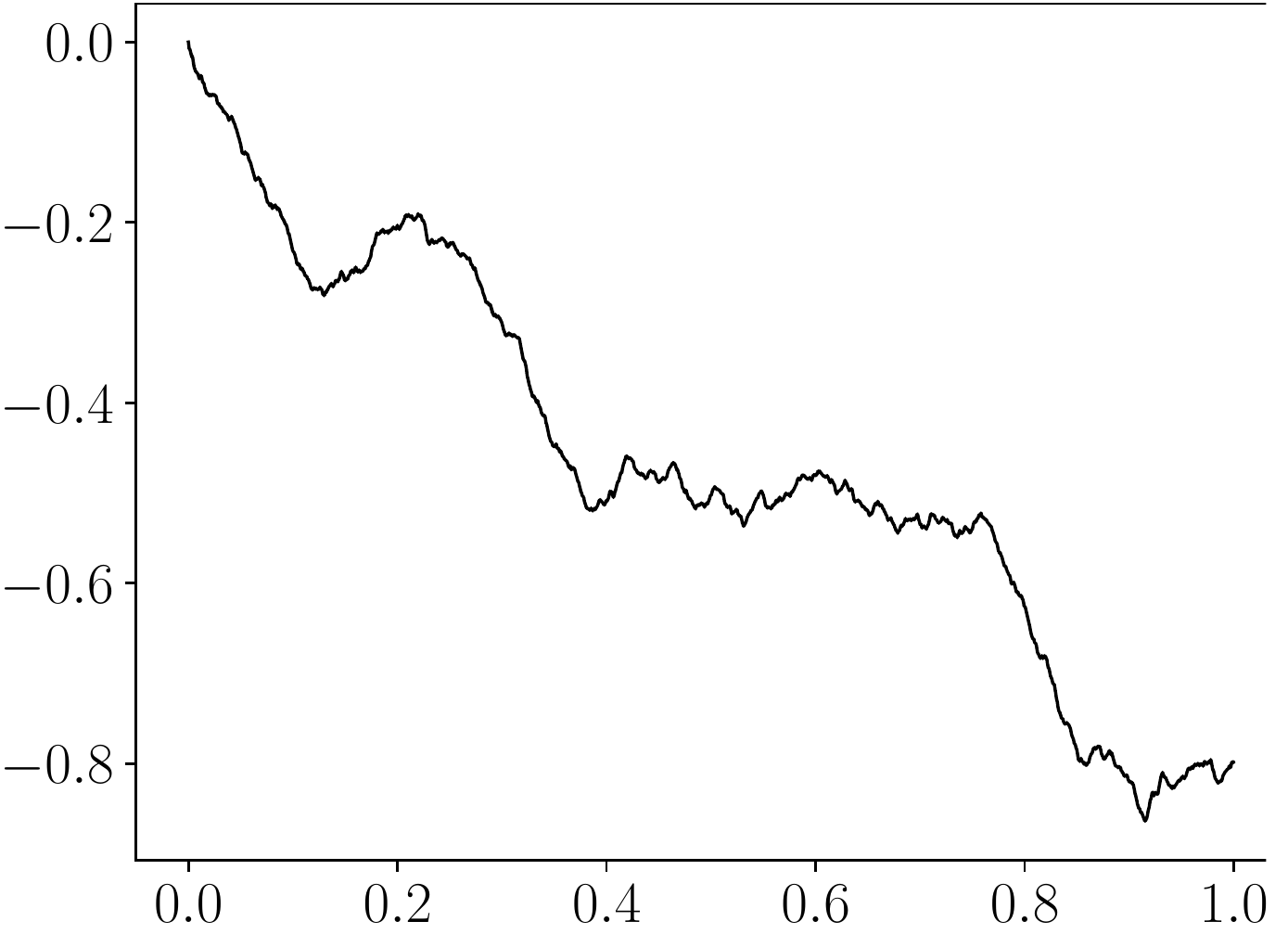} 
\end{tabular}
\caption{Three fractional Brownian motions $B_{.8}^{(1)}$, $B_{.8}^{(2)}$, and
$B_{.8}^{(3)}$.}
\label{brownian}
\end{figure}

Let $f : [0,1]^3 \rightarrow \mathbb{R}$ be defined by
\begin{equation} \label{fracbrownian}
f(x_1,x_2,x_3) = B_{.8}^{(1)}(x_1) B_{.8}^{(2)}(x_2) B_{.8}^{(3)}(x_3),
\end{equation}
for $(x_1,x_2,x_3) \in [0,1]^3$. Recall from the discussion in \S
\ref{motivation} that a tensor product of $\alpha$-H\"older continuous functions
on $[0,1]^d$ is an $\alpha$-mixed H\"older function on $[0,1]^d$. Since the
fractional Brownian motions $B_{.8}^{(1)}(x_1)$, $B_{.8}^{(2)}(x_2)$, and
$B_{.8}^{(3)}(x_3)$ are each almost surely $\alpha$-H\"older continuous for
$\alpha = .79$, it follows that the function $f(x_1,x_2,x_3)$ is almost surely
$\alpha$-mixed H\"older for the same value of $\alpha$.

For the numerical experiments, we set $c_1 = 3.5$ and sample $\sim c_1
\frac{1}{\varepsilon} (\log \frac{1}{\varepsilon})^3$ points uniformly at random
from $[0,1]^3$, where $\varepsilon = 2^{-m}$ for a given scale $m$. By sampling
the function at these random points, and using the algorithm detailed in \S
\ref{algo3} and \S \ref{algorand} we can construct an approximation $\tilde{f}$
that satisfies 
$$
\|f - \tilde{f}\|_{L^2} \lesssim \varepsilon^\alpha (\log
\textstyle{\frac{1}{\varepsilon}} )^{5/2},
$$
with probability at least $1 - \varepsilon^2$, see Theorem \ref{thm1}. For the
numerical experiments we estimate $\|f - \tilde{f}\|_{L^2}$ by using a random
test set. In particular, we sample $N = 10^8$ points $Y_1, \ldots,Y_N$ chosen
independently and uniformly at random from $[0,1]^3$, and defined $F :=
(f(Y_1),\ldots,f(Y_N)) \in \mathbb{R}^N$, $\tilde{F} =
(\tilde{f}(Y_1),\ldots,\tilde{f}(Y_N)) \in
\mathbb{R}^N$, and
$$
\text{err}_2 := \frac{\|F - \tilde{F}\|_2}{\|F\|_2}, 
$$
which can be viewed as a Monte Carlo approximation of the true relative error.
It will also be interesting to consider the maximum relative error
on the test set, which we define by
$$
\text{err}_\infty := \frac{\|F - \tilde{F}\|_\infty}{\|F\|_\infty}.
$$
Finally, by taking advantage of the product structure of $f$, we can compute the
integral $I$ of $f$ exactly. Let $\text{err}_\text{int} = |I -
\tilde{I}|/|I|$, where $\tilde{I}$ denotes the integral of our approximation
$\tilde{f}$, which can be efficiently computed as described in \S
\ref{compcost}.  We run the described numerical experiment for scales $m =
5,\ldots,18$ and report the values of $\text{err}_2$, $\text{err}_\infty$ and
$\text{err}_\text{int}$ in Table \ref{numtable}.

\begin{table}[h!]
$$
\begin{array}{c|c|c|c}
m & \text{err}_2 & \text{err}_\infty & \text{err}_\text{int}  \\
\hline 
5  & 2.1515E-01 & 2.8048E-01 & 1.1245E-02 \\
6  & 1.3670E-01 & 2.1546E-01 & 2.4038E-03 \\
7  & 9.1350E-02 & 1.5278E-01 & 2.7133E-04 \\
8  & 5.7977E-02 & 1.0867E-01 & 1.8701E-03 \\
9  & 3.6125E-02 & 6.8297E-02 & 3.1067E-04 \\
10 & 2.2403E-02 & 4.1824E-02 & 1.2265E-06 \\
11 & 1.3948E-02 & 3.6635E-02 & 2.9395E-05 \\
12 & 8.5529E-03 & 2.1643E-02 & 4.3197E-05 \\
13 & 5.2153E-03 & 1.6592E-02 & 1.2824E-05 \\
14 & 3.1656E-03 & 1.1169E-02 & 4.7572E-06 \\
15 & 1.9084E-03 & 6.7382E-03 & 1.1496E-06 \\
16 & 1.1459E-03 & 4.4197E-03 & 1.8777E-06 \\
17 & 6.8489E-04 & 2.3634E-03 & 3.5733E-07 \\
18 & 4.0723E-04 & 1.5863E-03 & 5.3804E-08
\end{array}
$$
\caption{Approximation errors for scales $m= 5,\ldots,18$.}
\label{numtable}
\end{table}

To provide context for the results reported in Table \ref{numtable} we make two
plots in Figure \ref{vis}. On the left, we plot $\text{err}_2$ and
$\text{err}_\infty$ along with two reference curves: first, we plot $\sim
\varepsilon^\alpha (\log \frac{1}{\varepsilon})^{5/2}$, which is the upper bound
from Theorem \ref{thm1} that holds with probability at least $1 -
\varepsilon^2$. Second, we plot $\sim \varepsilon^\alpha$ which is approximately
the amount that fractional Brownian motion with parameter $h = .8$ varies on an
interval of length $\varepsilon$, which provides a lower bound, see Figure
\ref{vis}. 

\begin{figure}[h!]
\begin{tabular}{cc}
\includegraphics[height=.32\textwidth]{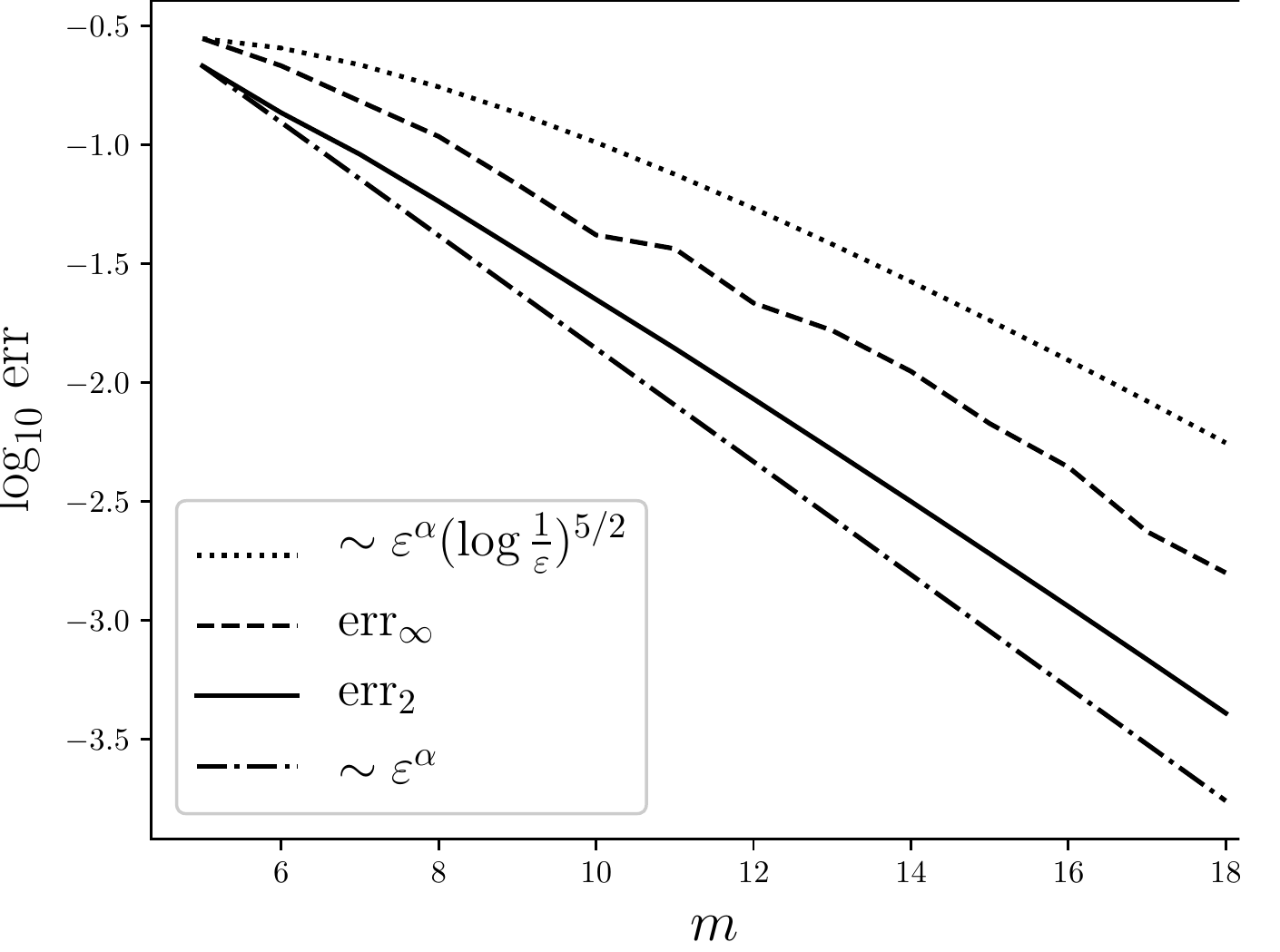} &
\includegraphics[height=.32\textwidth]{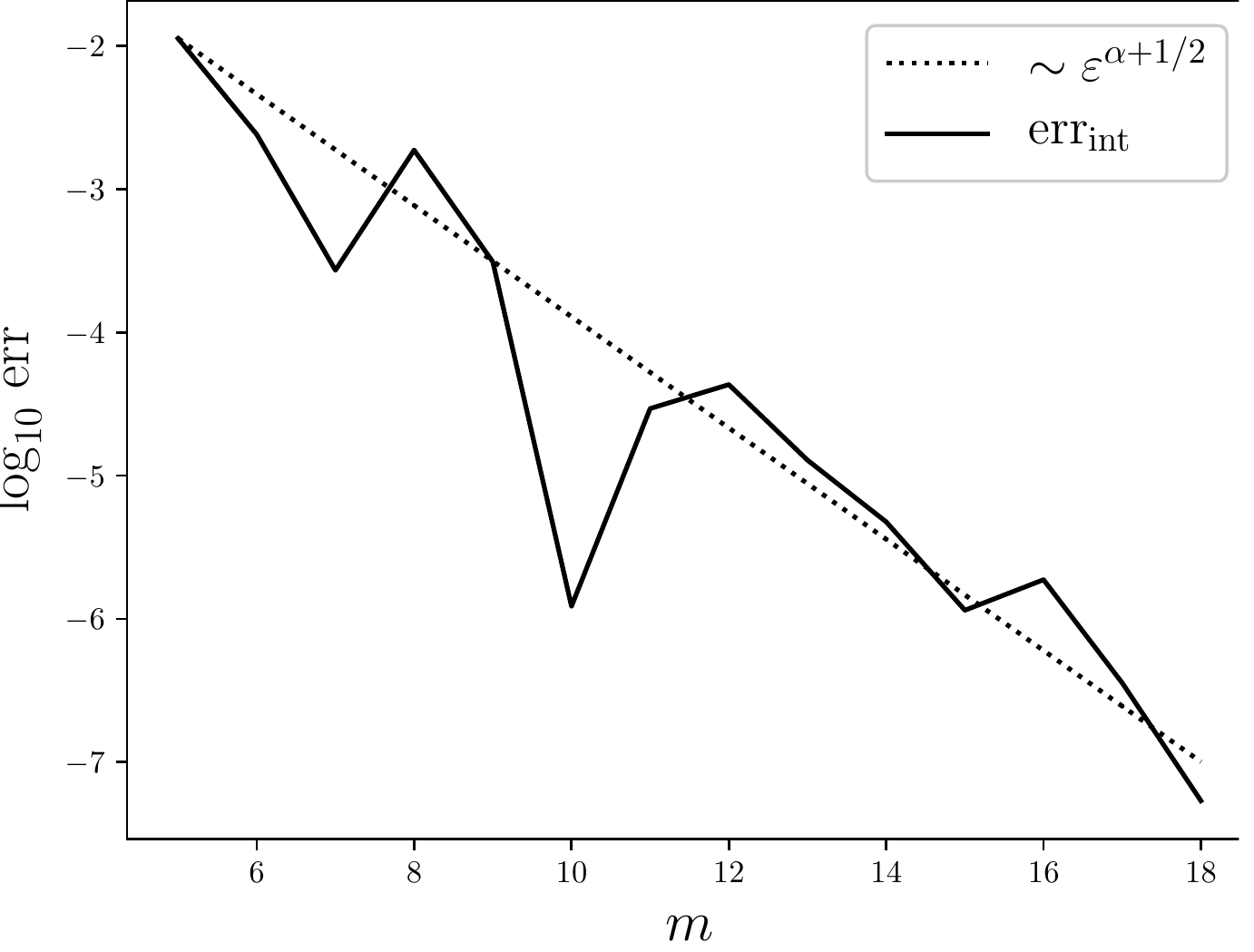}
\end{tabular}
\caption{Visualization of data in Table \ref{numtable} with references lines}
\label{vis}
\end{figure}

The numerical results suggest that the power on $\log
\frac{1}{\varepsilon}$ in Theorem \ref{thm1} can be improved, and suggest that
a similar approximation result may hold in $L^\infty$. 

In the right plot in Figure \ref{vis}, we
plot $\text{err}_\text{int}$ along with the reference curve $\sim
\varepsilon^{\alpha+1/2}$, which we suspect is the expected value of
$\text{err}_\text{int}$ up to factors of $\log \frac{1}{\varepsilon}$.  A
heuristic justification for the additional factor of $\varepsilon^{1/2}$
is that
computing the integral involves summing $\sim \frac{1}{\varepsilon}$ entries of
the vector $w_n^*$, which we might expect to have somewhat independent errors.
However, these errors are not statistically independent due to their interaction
in the iteration of the Kaczmarz algorithm.

\section{Discussion} \label{discussion}

In this paper, we developed a framework for synthesizing mixed H\"older
functions on $[0,1]^d$ from random samples of the function. Our principle analytical tool is
the embedding $\Psi : [0,1]^d \rightarrow \mathbb{R}^p$ that encodes the dyadic
boxes of measure at least $2^{-m}$ that contain a point. This embedding allows us
to consider a random sample of points $X_1,\ldots,X_n \in [0,1]^d$ as a random
matrix $ \left( \Psi(X_1) \cdots \Psi(X_n) \right) \in \mathbb{R}^{p \times n}$.
This encoding together with the randomized Kaczmarz algorithm provides a high
probability approximation method for mixed H\"older functions. This
method can be viewed as constructing a representation of a mixed H\"older
function in a tensor Haar wavelet expansion, which can be viewed as a local
product representation of a function. Product decompositions are common in the
field of data analysis as a way to represent matrices and tensors, for example,
the singular value decomposition represents a matrix as a sum of outer products
of vectors $A = \sigma_1 u_1 v_1^\top + \cdots + \sigma_k u_k v_k^\top.$ Part of
the larger goal of this paper is to understand how a function can be decomposed
and represented as a sum of local product structures. Recall that the class of
mixed H\"older functions is defined by the geometric condition that the mixed
difference with respect to each box is bounded by the measure of the box to
a fixed power $|\delta_R f| \lesssim |R|^\alpha$. This geometric condition is
well-defined for Banach space valued functions defined on a product of metric
spaces. For example, if metrics can be constructed on the rows and columns of a
matrix, or similarly on each of the dimensions of a tensor, then it is possible
to consider the class of mixed H\"older matrices or mixed H\"older tensors with
respect to the given metrics; this approach to matrix and tensor analysis was
initiated by Coifman and Gavish \cite{CoifmanGavish2011, GavishCoifman2012} and
has been developed by several authors including \cite{AnkenmanLeeb2018,
MishneCoifmanLavzinSchiller2018, MischneTalmonCohenCoifmanKluger2018,
MishneTalmonMeirSchillerDubinCoifman2016,
YairTalmonCoifmanKevrekidis2017}. The results of this paper for mixed H\"older
functions in the classical setting may provide insight for the development of
methods to represent or complete matrices or tensors. Moreover, in the classical
setting, it may be possible to generalize the results of this paper to
functions with higher levels of regularity. By considering mixed H\"older
functions rather than functions with higher levels of regularity, we were able
to use  tensor Haar wavelets rather than wavelets with more regularity; this
simplified our analysis and helped to isolate the underlying geometric issues.
If a function has more regularity, say, if the derivative of a function is mixed
H\"older, then it may be possible to develop a similar approach to that
described in this paper using tensor wavelets with more regularity to achieve
better error rates.

\subsection*{Acknowledgements}
The author would like to thank Ronald R. Coifman for many useful discussions.

\end{document}